\documentclass[12pt]{amsart}

\usepackage{color}
\usepackage{amsmath, amsthm, amssymb}
\usepackage{amsfonts}
\usepackage[ansinew]{inputenc}
\usepackage[dvips]{epsfig}
\usepackage{graphicx}
\usepackage[english]{babel}
\usepackage{hyperref}
\theoremstyle{plain}
\newtheorem{thm}{Theorem}[section]
\newtheorem{cor}[thm]{Corollary}
\newtheorem{lem}[thm]{Lemma}
\newtheorem{prop}[thm]{Proposition}
\newtheorem{claim}[thm]{Claim}

\theoremstyle{definition}
\newtheorem{defi}[thm]{Definition}

\theoremstyle{remark}
\newtheorem{rem}[thm]{Remark}

\numberwithin{equation}{section}

\newcommand{\average}{{\mathchoice {\kern1ex\vcenter{\hrule height.4pt
width 6pt depth0pt} \kern-9.7pt} {\kern1ex\vcenter{\hrule
height.4pt width 4.3pt depth0pt} \kern-7pt} {} {} }}

\def\R{\mathbb{R}}

\begin{document}

\title[The Dirichlet problem for the fractional Laplacian]{The Dirichlet problem for the fractional Laplacian: regularity up to the boundary}

\author{Xavier Ros-Oton}

\address{Universitat Polit\`ecnica de Catalunya, Departament de Matem\`{a}tica  Aplicada I, Diagonal 647, 08028 Barcelona, Spain}
\email{xavier.ros.oton@upc.edu}

\thanks{The authors were supported by grants MTM2008-06349-C03-01, MTM2011-27739-C04-01 (Spain), and 2009SGR345 (Catalunya)}

\author{Joaquim Serra}

\address{Universitat Polit\`ecnica de Catalunya, Departament de Matem\`{a}tica  Aplicada I, Diagonal 647, 08028 Barcelona, Spain}

\email{joaquim.serra@upc.edu}


\maketitle

\begin{abstract} We study the regularity up to the boundary of solutions to the Dirichlet problem for the fractional Laplacian. We prove that if $u$ is a solution of
$(-\Delta)^s u =g$ in $\Omega$,
$u\equiv0$ in $\mathbb R^n\backslash\Omega$,
for some $s\in(0,1)$ and $g\in L^\infty(\Omega)$,
then $u$ is $C^s(\R^n)$ and $u/\delta^s|_{\Omega}$ is $C^{\alpha}$ up to the boundary $\partial\Omega$ for some $\alpha\in(0,1)$, where $\delta(x)={\rm dist}(x,\partial\Omega)$.
For this, we develop a fractional analog of the Krylov boundary Harnack method.

Moreover, under further regularity assumptions on $g$ we obtain higher order H\"older estimates for $u$ and $u/\delta^s$. Namely, the $C^\beta$ norms of $u$ and $u/\delta^s$ in the sets $\{x\in\Omega:\delta(x)\geq\rho\}$ are controlled by $C\rho^{s-\beta}$ and $C\rho^{\alpha-\beta}$, respectively.

These regularity results are crucial tools in our proof of the Pohozaev identity for the fractional Laplacian \cite{RS-CRAS,RS}.
\end{abstract}

\section{Introduction and results}

Let $s\in(0,1)$ and $g\in L^\infty(\Omega)$, and consider the fractional elliptic problem
\begin{equation}\label{eqlin}
\left\{ \begin{array}{rcll} (-\Delta)^s u &=&g&\textrm{in }\Omega \\
u&=&0&\textrm{in }\mathbb R^n\backslash\Omega,\end{array}\right.
\end{equation}
in a bounded domain $\Omega\subset\mathbb R^n$, where
\begin{equation}
\label{laps}(-\Delta)^s u (x)= c_{n,s}{\rm PV}\int_{\R^n}\frac{u(x)-u(y)}{|x-y|^{n+2s}}dy
\end{equation}
and $c_{n,s}$ is a normalization constant.

Problem \eqref{eqlin} is the Dirichlet problem for the fractional Laplacian.
There are classical results in the literature dealing with the interior regularity of $s$-harmonic functions, or more generally for equations of the type \eqref{eqlin}.
However, there are few results on regularity up to the boundary. This is the topic of study of the paper.

Our main result establishes the H\"older regularity up to the boundary $\partial\Omega$ of the function $u/\delta^s|_\Omega$, where
\[\delta(x)={\rm dist}(x,\partial\Omega).\]
For this, we develop an analog of the Krylov \cite{Krylov} boundary Harnack method for problem \eqref{eqlin}.
As in Krylov's work, our proof applies also to operators with ``bounded measurable coefficients'',
more precisely those of the type \eqref{bmc}.
This will be treated in a future work \cite{RS-K}.
In this paper we only consider the constant coefficient operator $(-\Delta)^s$, since in this case we
can establish more precise regularity results.
Most of them will be needed in our subsequent work \cite{RS}, where we find and prove the
Pohozaev identity for the fractional Laplacian, announced in \cite{RS-CRAS}.
For \eqref{eqlin}, in addition to the H\"older regularity up to the boundary for $u/\delta^s$, we prove that any solution $u$ is $C^s(\R^n)$.
Moreover, when $g$ is not only bounded but H\"older continuous, we obtain better interior H\"older estimates for $u$ and $u/\delta^s$.

The Dirichlet problem for the fractional Laplacian \eqref{eqlin} has been studied from the point of view of probability, potential theory, and PDEs.
The closest result to the one in our paper is that of Bogdan \cite{B}, establishing a boundary Harnack inequality for nonnegative $s$-harmonic functions. It will be described in more detail later on in the Introduction (in relation with Theorem \ref{thm:v-is-Calpha}).
Related regularity results up to the boundary have been proved in \cite{KL} and \cite{CRS}.
In \cite{KL} it is proved that $u/\delta^s$ has a limit at every boundary point when $u$ solves the homogeneous fractional heat equation.
The same is proven in \cite{CRS} for a free boundary problem for the fractional Laplacian.

Some other results dealing with various aspects concerning the Dirichlet problem are the following:
estimates for the heat kernel (of the parabolic version of this problem) and for the Green function, e.g.,  \cite{BGR,CKS};
an explicit expression of the Poisson kernel for a ball \cite{L};
and the explicit solution to problem \eqref{eqlin} in a ball for $g\equiv1$ \cite{G}.
In addition, the interior regularity theory for viscosity solutions to nonlocal equations with ``bounded measurable coefficients'' is developed in \cite{CS}.

The first result of this paper gives the optimal H\"older regularity for a solution $u$ of \eqref{eqlin}.
The proof, which is given in Section \ref{sec3}, is based on two ingredients: a suitable upper barrier, and the  interior regularity results for the fractional Laplacian.
Given $g\in L^\infty(\Omega)$, we say that $u$ is a solution of \eqref{eqlin} when $u\in H^s(\R^n)$ is a weak  solution (see Definition \ref{weak}). When $g$ is continuous, the notions of weak solution and of viscosity solution agree; see Remark \ref{remviscosity}.

We recall that a domain $\Omega$ satisfies the exterior ball condition if there exists a positive radius $\rho_0$ such that all the points on $\partial \Omega$ can be touched by some exterior ball of radius $\rho_0$.

\begin{prop}\label{prop:u-is-Cs}
Let $\Omega$ be a bounded Lipschitz domain satisfying the exterior ball condition, $g\in L^\infty(\Omega)$, and $u$ be a solution of \eqref{eqlin}. Then, $u\in C^s(\R^n)$ and
\[\|u\|_{C^{s}(\R^n)}\le C \|g\|_{L^\infty (\Omega)},\]
where $C$ is a constant depending only on $\Omega$ and $s$.
\end{prop}

This $C^s$ regularity is optimal, in the sense that a solution to problem \eqref{eqlin} is not in general $C^\alpha$ for any $\alpha>s$.
This can be seen by looking at the problem
\begin{equation}\label{explicitsolution1}
\left\{ \begin{array}{rcll} (-\Delta)^s u &=&1&\textrm{in }B_r(x_0) \\
u&=&0&\textrm{in }\mathbb R^n\backslash B_r(x_0),\end{array}\right.\end{equation}
for which its solution is explicit. For any $r>0$ and $x_0\in\R^n$, it is given by \cite{G,BGR}
\begin{equation}\label{explicitsolution2}
u(x)=\frac{2^{-2s}\Gamma(n/2)}{\Gamma\left(\frac{n+2s}{2}\right)\Gamma(1+s)}
\left(r^2-|x-x_0|^2\right)^s\qquad\textrm{in}\ \ B_r(x_0).
\end{equation}
It is clear that this solution is $C^s$ up to the boundary but it is not $C^\alpha$ for any $\alpha>s$.

Since solutions $u$ of \eqref{eqlin} are $C^s$ up to the boundary, and not better, it is of importance to study the regularity of $u/\delta^s$ up to $\partial\Omega$.
For instance, our recent proof \cite{RS,RS-CRAS} of the Pohozaev identity for the fractional Laplacian uses in a crucial way that $u/\delta^s$ is H\"older continuous up to $\partial\Omega$.
This is the main result of the present paper and it is stated next.

For local equations of second order with bounded measurable coefficients and in non-divergence form, the analog result is given by a theorem of N. Krylov \cite{Krylov}, which states that
$u/\delta$ is $C^\alpha$ up to the boundary for some $\alpha\in (0,1)$.
This result is the key ingredient in the proof of the $C^{2,\alpha}$ boundary regularity of solutions to fully nonlinear elliptic equations $F(D^2u)=0$ ---see \cite{Kazdan,CaffC}.

For our nonlocal equation \eqref{eqlin}, the corresponding result is the following.

\begin{thm}\label{thm:v-is-Calpha}
Let $\Omega$ be a bounded $C^{1,1}$ domain, $g\in L^\infty(\Omega)$, $u$ be a solution of \eqref{eqlin}, and $\delta(x)={\rm dist}(x,\partial\Omega)$. Then, $u/\delta^s|_\Omega$ can be continuously extended to $\overline\Omega$. 
Moreover,  we have $u/\delta^s\in C^{\alpha}(\overline{\Omega})$ and
\[ \|u/\delta^s\|_{C^{\alpha}(\overline \Omega)}\le C \|g\|_{L^\infty(\Omega)}\]
for some $\alpha>0$ satisfying $\alpha<\min\{s,1-s\}$. 
The constants $\alpha$ and $C$ depend only on $\Omega$ and $s$.
\end{thm}

To prove this result we use the method of Krylov (see \cite{Kazdan}). It consists of trapping the solution between two multiples of $\delta^s$ in order to control the oscillation of the quotient $u/\delta^s$ near the boundary.
For this, we need to prove, among other things, that $(-\Delta)^s\delta_0^s$ is bounded in $\Omega$, where $\delta_0(x)={\rm dist}(x,\R^n\setminus\Omega)$ is the distance function in $\Omega$ extended by zero outside. This will be guaranteed by the assumption that $\Omega$ is $C^{1,1}$.

To our knowledge, the only previous results dealing with the regularity up to the boundary for solutions to \eqref{eqlin} or its parabolic version were the ones by K. Bogdan \cite{B} and S. Kim and K. Lee \cite{KL}.
The first one \cite{B} is the boundary Harnack principle for nonnegative $s$-harmonic functions, which reads as follows: assume that $u$ and $v$ are two nonnegative functions in a Lipschitz domain $\Omega$, which  satisfy $(-\Delta)^su\equiv0$ and $(-\Delta)^sv\equiv0$ in $\Omega\cap B_r(x_0)$ for some ball $B_r(x_0)$ centered at $x_0\in \partial\Omega$. Assume also that $u\equiv v\equiv 0$ in $B_r(x_0)\setminus\Omega$. Then, the quotient $u/v$ is $C^\alpha(\overline{B_{r/2}(x_0)})$ for some $\alpha\in(0,1)$.
In \cite{BKK} the same result is proven in open domains $\Omega$, without any regularity assumption.

While the result in \cite{BKK} assumes no regularity on the domain, we need to assume $\Omega$ to be $C^{1,1}$. This assumption is needed to compare the solutions with the function $\delta^s$.
As a counterpart, we allow nonzero right hand sides $g\in L^\infty(\Omega)$ and also changing-sign solutions.
In $C^{1,1}$ domains,
our results in Section \ref{sec4} (which are local near any boundary point) extend Bogdan's result.
For instance, assume that $u$ and $v$  satisfy $(-\Delta)^s u = g$ and $(-\Delta)^s v = h$ in $\Omega$, $u\equiv v\equiv 0$ in $\R^n\setminus \Omega$, and that $h$ is positive in $\Omega$. Then, by Theorem \ref{thm:v-is-Calpha} we have that $u/\delta^s$ and $v/\delta^s$ are $C^\alpha(\overline{\Omega})$ functions. In addition, by the Hopf lemma for the fractional Laplacian we find that $v/\delta^s\ge c>0$ in $\Omega$. Hence, we obtain that the quotient $u/v$ is $C^\alpha$ up to the boundary, as in Bogdan's result for $s$-harmonic functions.

As in Krylov's result, our method can be adapted
to the case of nonlocal elliptic equations with ``bounded measurable coefficients''.
Namely, in another paper \cite{RS-K} we will prove the boundary Harnack principle for solutions to $\mathcal{L}u=g$ in $\Omega$, $u\equiv0$ in $\R^n\setminus\Omega$, where $g\in L^\infty(\Omega)$,
\begin{equation}\label{bmc}
\mathcal{L}u(x)=\int_{\R^n}\frac{2u(x)-u(x+y)-u(x-y)}{\left|y^TA(x)y\right|^{\frac{n+2s}{2}}}dy,
\end{equation}
and $A(x)$ is a symmetric matrix, measurable in $x$, and with $0<\lambda{\rm Id}\leq A(x)\leq \Lambda{\rm Id}$.

A second result (for the parabolic problem) related to ours is contained in \cite{KL}. The authors show that any solution of $\partial_t u + (-\Delta)^s u =0$ in $\Omega$, $u\equiv 0$ in $\R^n\setminus \Omega$, satisfies the following property: for any $t>0$ the function $u/\delta^s$ is continuous up to the boundary $\partial\Omega$.

Our results were motivated by the study of nonlocal semilinear problems $(-\Delta)^su=f(u)$ in $\Omega$, $u\equiv0$ in $\R^n\setminus\Omega$,
more specifically, by the Pohozaev identity that we establish in \cite{RS}.
Its proof requires the precise regularity theory up to the boundary developed in the present paper (see Corollary \ref{krylov} below).
Other works treating the fractional Dirichlet semilinear problem, which deal mainly with existence of solutions
and symmetry properties, are \cite{SV,ServV,FW,BMW}.

In the semilinear case, $g=f(u)$ and therefore $g$ automatically becomes more regular than just bounded. When $g$ has better regularity,
the next two results improve the preceding ones.
The proofs of these results require the use of the following weighted H\"older norms, a slight modification of the ones in Gilbarg-Trudinger \cite[Section 6.1]{GT}.

Throughout the paper, and when no confusion is possible, we use the notation $C^\beta(U)$ with $\beta>0$ to refer to the space $C^{k,\beta'}(U)$, where $k$ is the is greatest integer such that $k<\beta$ and where $\beta'=\beta-k$. This notation is specially appropriate when we work with $(-\Delta)^s$ in order to avoid the splitting of different cases in the statements of regularity results.
According to this, $[\,\cdot\,]_{C^{\beta}(U)}$ denotes the $C^{k,\beta'}(U)$ seminorm
\[[u]_{C^\beta(U)}=[u]_{C^{k,\beta'}(U)}=\sup_{x,y\in U,\ x\neq y}\frac{|D^ku(x)-D^ku(y)|}{|x-y|^{\beta'}}.\]
Moreover, given an open set $U\subset\R^n$ with $\partial U\neq\varnothing$, we will also denote
\[d_x= \mathrm{dist}(x,\partial U) \qquad\mbox{and}\qquad d_{x,y}=\min\{d_x,d_y\}.\]

\begin{defi}\label{definorm} Let $\beta>0$ and $\sigma\ge -\beta$. Let $\beta=k+\beta'$, with $k$ integer and $\beta'\in (0,1]$.
For $w\in C^{\beta}(U)=C^{k,\beta'}(U)$, define the seminorm
\[ [w]_{\beta;U}^{(\sigma)}= \sup_{x,y\in U} \biggl(d_{x,y}^{\beta+\sigma} \frac{|D^{k}w(x)-D^{k}w(y)|}{|x-y|^{\beta'}}\biggr).\]
For $\sigma>-1$, we also define  the norm $\|\,\cdot\,\|_{\beta;U}^{(\sigma)}$ as follows: in case that $\sigma\ge0$,
\[ \|w\|_{\beta;U}^{(\sigma)} = \sum_{l=0}^k \sup_{x\in U} \biggl(d_x^{l+\sigma} |D^l w(x)|\biggr) + [w]_{\beta;U}^{(\sigma)}\,,\]
while for $-1<\sigma<0$,
\[\|w\|_{\beta;U}^{(\sigma)} = \|w\|_{C^{-\sigma}(\overline U)}+\sum_{l=1}^k \sup_{x\in U} \biggl(d_x^{l+\sigma} |D^l w(x)|\biggr) + [w]_{\beta;U}^{(\sigma)}.\]

Note that $\sigma$ is the rescale order of the seminorm $[\,\cdot\,]_{\beta;U}^{(\sigma)}$, in the sense that $[w(\lambda\cdot)]_{\beta;U/\lambda}^{(\sigma)} = \lambda^\sigma[w]_{\beta;U}^{(\sigma)}$.
\end{defi}

When $g$ is H\"older continuous, the next result provides optimal estimates for higher order H\"older norms of $u$ up to the boundary.

\begin{prop}\label{prop:int-est-u}
Let $\Omega$ be a bounded domain, and $\beta>0$ be such that neither $\beta$ nor $\beta+2s$ is an integer. Let $g\in C^\beta(\Omega)$ be such that $\|g\|_{\beta;\Omega}^{(s)}<\infty$, and $u\in C^s(\R^n)$ be a solution of \eqref{eqlin}. Then,
$u\in C^{\beta+2s}(\Omega)$ and
\[\|u\|_{\beta+2s;\Omega}^{(-s)}\le C \bigl(\|u\|_{C^s(\R^n)}+\|g\|_{\beta;\Omega}^{(s)}\bigr),\]
where $C$ is a constant depending only on $\Omega$, $s$, and $\beta$.
\end{prop}

Next, the H\"older regularity up to the boundary of $u/\delta^s$ in Theorem \ref{thm:v-is-Calpha} can be improved when $g$ is H\"older continuous. This is stated in the following theorem, whose proof uses a nonlocal equation satisfied by the quotient $u/\delta^s$ in $\Omega$ ---see \eqref{equaciov}--- and the fact that this quotient is $C^\alpha(\overline\Omega)$.

\begin{thm}\label{thm:int-est-v}
Let $\Omega$ be a bounded $C^{1,1}$ domain, and let $\alpha\in(0,1)$ be given by Theorem \ref{thm:v-is-Calpha}. Let $g\in L^\infty(\Omega)$ be such that $\|g\|_{\alpha;\Omega}^{(s-\alpha)}<\infty$, and $u$ be a solution of \eqref{eqlin}.
Then, $u/\delta^s\in C^\alpha(\overline\Omega)\cap C^\gamma(\Omega)$ and
\[ \|u/\delta^s\|_{\gamma;\Omega}^{(-\alpha)}\le C \bigl(\|g\|_{L^\infty(\Omega)}+\|g\|_{\alpha;\Omega}^{(s-\alpha)}\bigr),\]
where $\gamma=\min\{1,\alpha+2s\}$ and $C$ is a constant depending only on $\Omega$ and $s$.
\end{thm}

Finally, we apply the previous results to the semilinear problem
\begin{equation}\label{eqnonlin}
\left\{ \begin{array}{rcll} (-\Delta)^s u &=&f(x,u)&\textrm{in }\Omega \\
u&=&0&\textrm{on }\mathbb R^n\backslash\Omega,\end{array}\right.
\end{equation}
where $\Omega$ is a bounded $C^{1,1}$ domain and $f$ is a Lipschitz nonlinearity.

In the following result, the meaning of ``bounded solution'' is that of ``bounded weak solution'' (see definition \ref{weak}) or that of ``viscosity solution''.
By Remark \ref{remviscosity}, these two notions coincide.
Also, by $f\in C^{0,1}_{\rm loc}(\overline\Omega\times\R)$ we mean that $f$ is Lipschitz in every compact subset of $\overline\Omega\times\R$.

\begin{cor}\label{krylov} Let $\Omega$ be a bounded and $C^{1,1}$ domain, $f\in C^{0,1}_{\rm loc}(\overline\Omega\times\R)$, $u$ be a bounded solution of
\eqref{eqnonlin}, and $\delta(x)={\rm dist}(x,\partial\Omega)$.
Then,
\begin{itemize}
\item[(a)] $u\in C^s(\R^n)$ and, for every $\beta\in[s,1+2s)$, $u$ is of class $C^{\beta}(\Omega)$ and
\[[u]_{C^{\beta}(\{x\in\Omega\,:\,\delta(x)\ge\rho\})}\le C \rho^{s-\beta}\qquad \textrm{for all}\ \  \rho\in(0,1).\]
\item[(b)] The function $u/\delta^s|_\Omega$ can be continuously extended to $\overline\Omega$. Moreover, there exists $\alpha\in(0,1)$ such that $u/\delta^s\in C^{\alpha}(\overline{\Omega})$. In addition, for all $\beta\in[\alpha,s+\alpha]$, it holds the estimate
\[ [u/\delta^s]_{C^{\beta}(\{x\in\Omega\,:\,\delta(x)\ge\rho\})}\le C \rho^{\alpha-\beta}\qquad \textrm{for all}\ \  \rho\in(0,1).\]
\end{itemize}
The constants $\alpha$ and $C$ depend only on $\Omega$, $s$, $f$, $\|u\|_{L^{\infty}(\R^n)}$, and $\beta$.
\end{cor}

The paper is organized as follows.
In Section \ref{sec3} we prove Propositions \ref{prop:u-is-Cs} and \ref{prop:int-est-u}.
In Section \ref{sec4} we prove Theorem \ref{thm:v-is-Calpha} using the Krylov method.
In Section \ref{sec5} we prove Theorem \ref{thm:int-est-v} and Corollary \ref{krylov}.
Finally, the Appendix deals with some basic tools and barriers which are used throughout the paper.

\section{Optimal H\"older regularity for $u$}
\label{sec3}

In this section we prove that, assuming $\Omega$ to be a  bounded Lipschitz domain satisfying the exterior ball condition, every solution $u$ of \eqref{eqlin} belongs to $C^s(\R^n)$.
For this, we first establish that $u$ is $C^{\beta}$ in $\Omega$, for all $\beta\in(0,2s)$, and sharp bounds for the corresponding seminorms near $\partial\Omega$.
These bounds yield $u\in C^s(\R^n)$ as a corollary.
First, we make precise the notion of weak solution to problem \eqref{eqlin}.

\begin{defi}\label{weak} We say that $u$ is a weak solution of \eqref{eqlin} if $u\in H^s(\R^n)$, $u\equiv 0$ (a.e.) in $\R^n\setminus\Omega$, and
\[\int_{\mathbb R^n}(-\Delta)^{s/2}u(-\Delta)^{s/2}v\,dx=\int_\Omega gv\,dx\]
for all $v\in H^s(\R^n)$ such that $v\equiv0$ in $\R^n\setminus\Omega$.
\end{defi}

We recall first some well known interior regularity results for linear equations involving the operator $(-\Delta)^s$, defined by \eqref{laps}.
The first one states that $w\in C^{\beta+2s}(\overline{B_{1/2}})$ whenever $w\in C^\beta(\R^n)$ and $(-\Delta)^sw\in C^\beta(\overline{B_1})$.
Recall that, throughout this section and in all the paper, we denote by $C^\beta$, with $\beta>0$, the space $C^{k,\beta'}$, where $k$ is an integer, $\beta'\in(0,1]$, and $\beta=k+\beta'$.

\begin{prop}\label{prop-reg-lin-2}
Assume that $w\in C^\infty (\R^n)$ solves $(-\Delta)^s w = h$ in $B_1$ and that neither $\beta$ nor $\beta+2s$ is an integer. Then,
\[ \|w\|_{C^{\beta+2s}(\overline{B_{1/2}})} \le C\bigl( \|w\|_{C^\beta(\R^n)}+ \|h\|_{C^\beta(\overline{B_1})}\bigr)\,,\]
where $C$ is a constant depending only on $n$, $s$, and $\beta$.
\end{prop}

\begin{proof} Follow the proof of Proposition 2.1.8 in \cite{S}, where the same result is proved with $B_1$ and $B_{1/2}$ replaced by the whole $\R^n$.
\end{proof}

The second result states that $w\in C^{\beta}(\overline{B_{1/2}})$ for each $\beta\in(0,2s)$ whenever $w\in L^\infty(\R^n)$ and $(-\Delta)^sw\in L^\infty({B_1})$.

\begin{prop} \label{prop-reg-lin-1}
Assume that $w\in C^\infty (\R^n)$ solves $(-\Delta)^s w = h$ in $B_1$.
Then, for every $\beta\in(0,2s)$,
\[ \|w\|_{C^{\beta}(\overline{B_{1/2}})} \le C\bigl( \|w\|_{L^\infty(\R^n)}+ \|h\|_{L^\infty(B_1)}\bigr)\,,\]
where $C$ is a constant depending only on $n$, $s$, and $\beta$.
\end{prop}

\begin{proof}
Follow the proof of Proposition 2.1.9 in \cite{S}, where the same result is proved in the whole $\R^n$.
\end{proof}

The third result is the analog of the first, with the difference that it does not need to assume $w\in C^\beta(\R^n)$, but only $w\in C^\beta(\overline{B_2})$ and
$(1+|x|)^{-n-2s}w(x)\in L^1(\R^n)$.

\begin{cor}\label{int-est-brick}
Assume that $w\in C^{\infty}(\R^n)$ is a solution of $(-\Delta)^s w = h$ in $B_2$, and that neither $\beta$ nor $\beta+2s$ is an integer.
Then,
\[ \|w\|_{C^{\beta+2s}(\overline{B_{1/2}})} \le C\biggl(\|(1+|x|)^{-n-2s}w(x)\|_{L^1(\R^n)} + \|w\|_{C^{\beta}(\overline {B_2})}  + \|h\|_{C^{\beta}(\overline{B_2})}  \biggr)\]
where the constant $C$ depends only on $n$, $s$, and $\beta$.
\end{cor}

\begin{proof}
Let $\eta\in C^\infty(\R^n)$ be such that $\eta\equiv 0$ outside $B_2$ and $\eta\equiv 1$ in $B_{3/2}$. Then $\tilde w:= w\eta\in C^{\infty}(\R^n)$ and $(-\Delta)^s \tilde w = \tilde h := h - (-\Delta)^{s} \bigl(w(1-\eta)\bigr)$.
Note that for $x\in B_{3/2}$ we have
\[(-\Delta)^s \left(w(1-\eta)\right)(x) = c_{n,s}\int_{\R^n\setminus{B_{3/2}}}\frac{ - \bigl(w(1-\eta)\bigr)(y)}{|x-y|^{n+2s}}dy.\]
From this expression we obtain that
\[ \|(-\Delta)^s \left(w(1-\eta)\right) \|_{L^\infty({B_1})}\le C \|(1+|y|)^{-n-2s} w(y)\|_{L^1(\R^n)}\]
and for all $\gamma\in(0,\beta]$,
\[
\begin{split}
[(-\Delta)^s \left(w(1-\eta)\right) ]_{C^\gamma(\overline{B_1})}&\leq C \|(1+|y|)^{-n-2s-\gamma} w(y)\|_{L^1(\R^n)}\\&\le C \|(1+|y|)^{-n-2s} w(y)\|_{L^1(\R^n)}
\end{split}
\]
for some constant $C$ that depends only on $n$, $s$, $\beta$, and $\eta$. Therefore
\[ \|\tilde h\|_{C^{\beta}(\overline {B_{1}})}\le C \bigl(\|h\|_{C^{\beta}(\overline{B_2})} +  \|(1+|x|)^{-n-2s} w(x)\|_{L^1(\R^n)}\bigr), \]
while we also clearly have
\[
\|\tilde w\|_{C^{\beta}(\overline{\R^n})}\le  C\|w\|_{C^{\beta}(\overline{B_2})}\,.
\]
The constants $C$ depend only on $n$, $s$, $\beta$ and $\eta$. Now, we finish the proof by applying Proposition \ref{prop-reg-lin-2} with $w$ replaced by $\tilde w$.
\end{proof}

Finally, the fourth result is the analog of the second one, but instead of assuming $w\in L^\infty(\R^n)$, it only assumes $w\in L^\infty(B_2)$ and $(1+|x|)^{-n-2s}w(x)\in L^1(\R^n)$.

\begin{cor}\label{int-est-brick2}
Assume that $w\in C^\infty(\R^n)$ is a solution of $(-\Delta)^s w = h$ in $B_2$.
Then, for every $\beta\in(0,2s)$,
\[ \|w\|_{C^{\beta}(\overline{B_{1/2}})} \le C\biggl(\|(1+|x|)^{-n-2s}w(x)\|_{L^1(\R^n)} + \|w\|_{L^\infty(B_2)}  + \|h\|_{L^\infty(B_2)}  \biggr)\]
where the constant $C$ depends only on $n$, $s$, and $\beta$.
\end{cor}

\begin{proof}
Analog to the proof of Corollary \ref{int-est-brick}.
\end{proof}

As a consequence of the previous results we next prove that every solution $u$ of \eqref{eqlin} is $C^s(\R^n)$.
First let us find an explicit upper barrier for $|u|$ to prove that $|u|\le C \delta^s$ in $\Omega$.
This is the first step to obtain the $C^s$ regularity.

To construct this we will need the following result, which is proved in the Appendix.

\begin{lem}[Supersolution]\label{prop:supersolution}
There exist $C_1>0$ and a radial continuous function $\varphi_1\in H^s_{\rm loc}(\R^n)$ satisfying
\begin{equation}\label{eq:propsupersol}
\begin{cases}
(-\Delta)^s \varphi_1 \ge 1 &\mbox{in }B_4\setminus B_1\\
\varphi_1 \equiv 0 \quad &\mbox{in }B_1 \\
0\le\varphi_1 \le C_1(|x|-1)^s &\mbox{in }B_4\setminus B_1\\
1\le \varphi_1 \le C_1 &\mbox{in }\R^n\setminus B_4\,.
\end{cases}
\end{equation}
\end{lem}

The upper barrier for $|u|$ will be constructed by scaling and translating the supersolution from Lemma \ref{prop:supersolution}.
The conclusion of this barrier argument is the following.

\begin{lem}\label{lem-aux-Cs}
Let $\Omega$ be a bounded domain satisfying the exterior ball condition and let $g\in L^\infty(\Omega)$. Let $u$ be the solution of \eqref{eqlin}. Then,
\[ |u(x)| \le C\|g\|_{L^{\infty}(\Omega)} \delta^s(x)\quad \mbox{for all }x\in\Omega\,,\]
where $C$ is a constant depending only on $\Omega$ and $s$.
\end{lem}

In the proof of Lemma \ref{lem-aux-Cs} it will be useful the following

\begin{claim}\label{Linftybound} Let $\Omega$ be a bounded domain and let $g\in L^\infty(\Omega)$. Let $u$ be the solution of \eqref{eqlin}. Then,
\[ \|u\|_{L^\infty(\R^n)} \le C ({\rm diam}\,\Omega)^{2s}\|g\|_{L^{\infty}(\Omega)} \]
where $C$ is a constant depending only on $n$ and $s$.
\end{claim}

\begin{proof}
The domain $\Omega$ is contained in a large ball of radius ${\rm diam}\,\Omega$. Then, by scaling the explicit (super)solution for the ball given by \eqref{explicitsolution2} we obtain the desired bound.
\end{proof}

We next give the

\begin{proof}[Proof of Lemma \ref{lem-aux-Cs}]
Since $\Omega$ satisfies the exterior ball condition, there exists $\rho_0>0$ such that
every point of $\partial\Omega$ can be touched from outside by a ball of radius $\rho_0$.
Then, by scaling and translating the supersolution $\varphi_1$ from Lemma \ref{prop:supersolution}, for each of this exterior tangent balls $B_{\rho_0}$ we find an upper barrier in $B_{2\rho_0}\setminus B_{\rho_0}$ vanishing in $\overline{B_{\rho_0}}$.
This yields the bound $u\le C\delta^s$ in a $\rho_0$-neighborhood of $\partial\Omega$. By using Claim \ref{Linftybound} we have the same bound in all of $\overline\Omega$. Repeating the same argument with $-u$ we find $|u|\le C\delta^s$, as wanted.
\end{proof}

The following lemma gives interior estimates for $u$ and yields, as a corollary, that every bounded weak solution $u$ of \eqref{eqlin} in a $C^{1,1}$ domain is $C^s(\R^n)$.

\begin{lem}\label{lem-sarp-Cs-bounds-u}
Let $\Omega$ be a bounded domain satisfying the exterior ball condition, $g\in L^\infty(\Omega)$, and $u$ be the solution of \eqref{eqlin}. Then, $u \in C^\beta(\Omega)$ for all $\beta\in (0,2s)$ and for all $x_0\in \Omega$ we have the following seminorm estimate in $B_R(x_0)=B_{\delta(x_0)/2}(x_0)$:
\begin{equation}\label{first-seminorm-estimate-u}
[u]_{C^\beta(\overline{B_{R}(x_0)})}\le C R^{s-\beta}\|g\|_{L^{\infty}(\Omega)},
\end{equation}
where $C$ is a constant depending only on $\Omega$, $s$, and $\beta$.
\end{lem}

\begin{proof} Recall that if $u$ solves \eqref{eqlin} in the weak sense and $\eta_\epsilon$ is the standard mollifier then $(-\Delta)^s(u\ast \eta_\epsilon)=g\ast \eta_\epsilon$ in $B_R$ for $\epsilon$ small enough. Hence, we can regularize $u$, obtain the estimates, and then pass to the limit. In this way we may assume that $u$ is smooth.

Note that $B_R(x_0)\subset B_{2R}(x_0)\subset \Omega$. Let $\tilde u(y)= u(x_0+Ry)$. We have that
\begin{equation}\label{Rs1}
(-\Delta)^s \tilde u(y) =  R^{2s} g(x_0+Ry)\quad \mbox{in }  B_1\,.
\end{equation}
Furthermore, using that $|u| \le C\bigl(\|u\|_{L^{\infty}(\R^n)}+ \|g\|_{L^{\infty}(\Omega)}\bigr) \delta^s$ in $\Omega$ ---by Lemma \ref{lem-aux-Cs}--- we obtain
\begin{equation}\label{Rs2}
\|\tilde u\|_{L^\infty(B_1)}\le C\bigl(\|u\|_{L^{\infty}(\R^n)}+ \|g\|_{L^{\infty}(\Omega)}\bigr)  R^s
\end{equation}
and, observing that $|\tilde u(y)|\le C\bigl(\|u\|_{L^{\infty}(\R^n)}+ \|g\|_{L^{\infty}(\Omega)}\bigr) R^s(1+|y|^s)$ in all of $\R^n$,
\begin{equation}\label{Rs3}
\|(1+|y|)^{-n-2s}\tilde u(y)\|_{L^1(\R^n)}\le C \bigl(\|u\|_{L^{\infty}(\R^n)}+ \|g\|_{L^{\infty}(\Omega)}\bigr)  R^s,
\end{equation}
with $C$ depending only on $\Omega$ and $s$.

Next we use Corollary \ref{int-est-brick2}, which taking into account \eqref{Rs1}, \eqref{Rs2}, and \eqref{Rs3}, yields
\[
\|\tilde u\|_{C^{\beta}\left(\overline{B_{1/4}}\right)} \le C \bigl(\|u\|_{L^{\infty}(\R^n)}+ \|g\|_{L^{\infty}(\Omega)}\bigr)R^s
\]
for all $\beta\in(0,2s)$, where $C=C(\Omega,s,\beta)$.

Finally, we observe that
\[[u]_{C^\beta\left(\overline{B_{R/4}(x_0)}\right)}=R^{-\beta}[\tilde u]_{C^\beta\left(\overline{B_{1/4}}\right)}.\]
Hence, by an standard covering argument, we find the estimate \eqref{first-seminorm-estimate-u} for the $C^\beta$ seminorm of $u$ in $\overline {B_{R}(x_0)}$.
\end{proof}

We now prove the $C^s$ regularity of $u$.

\begin{proof}[Proof of Proposition \ref{prop:u-is-Cs}]
By Lemma \ref{lem-sarp-Cs-bounds-u}, taking $\beta=s$ we obtain
\begin{equation}\label{cotaCsenboles}
\frac{|u(x)-u(y)|}{|x-y|^s}\leq C\bigl( \|u\|_{L^\infty(\R^n)}+ \|g\|_{L^\infty(\Omega)}\bigr)
\end{equation}
for all $x,y$ such that $y\in B_R(x)$ with $R=\delta(x)/2$. We want to show that \eqref{cotaCsenboles} holds, perhaps with a bigger constant $C= C(\Omega,s)$, for all $x,y\in \overline\Omega$, and hence for all $x,y\in \R^n$ (since $u\equiv 0$ outside $\Omega$).

Indeed, observe that after a Lipschitz change of coordinates, the bound \eqref{cotaCsenboles} remains the same except for the value of the constant $C$. Hence, we can flatten the boundary near $x_0\in\partial\Omega$ to assume that $\Omega\cap B_{\rho_0}(x_0)=\{x_n>0\}\cap B_1(0)$. Now, \eqref{cotaCsenboles} holds for all $x,y$ satisfying $|x-y|\le \gamma x_n $ for some $\gamma =\gamma(\Omega)\in (0,1)$ depending on the Lipschitz map.

Next, let $z=(z',z_n)$ and $w=(w',w_n)$ be two points in $\{x_n>0\}\cap B_{1/4}(0)$, and $r=|z-w|$. Let us define $\bar z=(z',z_n+r)$, $\bar z=(z',z_n+r)$ and $z_k = (1-\gamma^k) z + \gamma^k\bar z$ and $w_k = \gamma^k w + (1-\gamma^k)\bar w$, $k\ge0$.
Then, using that bound \eqref{cotaCsenboles} holds whenever $|x-y|\le \gamma x_n $, we have
\[|u(z_{k+1})-u(z_{k})|\le C|z_{k+1}-z_k|^s = C|\gamma^{k}(z-\bar z)(\gamma-1)|^s\leq C\gamma^k|z-\bar z|. \]
Moreover, since $x_n>r$ in all the segment joining $\bar z$ and $\bar w$, splitting this segment into a bounded number of segments of length less than $\gamma r$, we obtain
\[|u(\bar z)-u(\bar w)|\leq C|\bar z-\bar w|^s\leq Cr^s.\]
Therefore,
\[
\begin{split}|u(z)-u(w)|&\leq \sum_{k\ge 0} |u(z_{k+1})-u(z_{k})|+|u(\bar z)-u(\bar w)|+\sum_{k\geq0} |u(w_{k+1})-u(w_{k})|\\
&\leq \left( C \sum_{k\ge 0} \bigl(\gamma^{k} r\bigr)^s + C r^s \right) \bigl( \|u\|_{L^\infty(\R^n)}+ \|g\|_{L^\infty(\Omega)}\bigr)\\
&\le C\bigl( \|u\|_{L^\infty(\R^n)}+ \|g\|_{L^\infty(\Omega)}\bigr) |z-w|^s,
\end{split}
\]
as wanted.
\end{proof}

The following lemma is similar to Proposition \ref{prop-reg-lin-2} but it involves the weighted norms introduced above. It will be used to prove Proposition \ref{prop:int-est-u} and Theorem \ref{thm:int-est-v}.

\begin{lem} \label{refined-2s-gain}
Let $s$ and $\alpha$ belong to $(0,1)$, and $\beta>0$. Let $U$ be an open set with nonempty boundary. Assume that neither $\beta$ nor $\beta+2s$ is an integer, and $\alpha<2s$. Then,
\begin{equation}\label{eq:2s-derivatives-more}
\|w\|_{\beta+2s;U}^{(-\alpha)}\le C\biggl( \|w\|_{C^{\alpha}(\R^n)}+ \|(-\Delta)^s w\|_{\beta;U}^{(2s-\alpha)}\biggr)
\end{equation}
for all $w$ with finite right hand side.
The constant $C$ depends only on $n$, $s$, $\alpha$, and $\beta$.
\end{lem}

\begin{proof} \emph{Step 1.} We first control the $C^{\beta+2s}$ norm of $w$ in balls $B_R(x_0)$ with $R=d_{x_0}/2$.

Let $x_0\in U$ and $R=d_{x_0}/2$. Define $\tilde w(y)= w(x_0+Ry)-w(x_0)$ and note that
\[ \|\tilde w\|_{C^\alpha(B_1)}\le R^\alpha [w]_{C^{\alpha}(\R^n)}\]
and
\[  \|(1+|y|)^{-n-2s} \tilde w(y)\|_{L^1(\R^n)} \le C(n,s) R^\alpha [w]_{C^{\alpha}(\R^n)}.\]
This is because
\[|\tilde w(y)| = |w(x_0+Ry)-w(x_0)|\le R^\alpha |y|^\alpha [w]_{C^{\alpha}(\R^n)}\]
and $\alpha< 2s$. Note also that
\[ \|(-\Delta)^s\tilde w\|_{C^\beta(\overline{B_1})}= R^{2s+\beta}\|(-\Delta)^s w\|_{C^\beta(\overline{B_R(x_0)})} \le R^\alpha\|(-\Delta)^s w\|_{\beta;U}^{(2s-\alpha)}\,.\]

Therefore, using Corollary \ref{int-est-brick} we obtain that
\[ \|\tilde w\|_{C^{\beta+2s}(\overline{B_{1/2}})}  \le C R^\alpha\bigl([w]_{C^{\alpha}(\R^n)}+ \|(-\Delta)^s w\|_{\beta;U}^{(2s-\alpha)}\bigr),\]
where the constant $C$ depends only on $n$, $s$, $\alpha$, and $\beta$. Scaling back we obtain
\begin{equation}\label{dos}
\begin{split}
\sum_{l=1}^k R^{l-\alpha} \|D^l w\|_{L^\infty(B_{R/2}(x_0))} + R^{2s+\beta-\alpha} [w]_{C^{\beta+2s}(\overline{B_{R/2}(x_0)})}\le \\  \le C\bigl(\|w\|_{C^{\alpha}(\R^n)}+  \|(-\Delta)^s w\|_{\alpha;U}^{(2s-\alpha)}\bigr),
\end{split}
\end{equation}
where $k$ denotes the greatest integer less that $\beta+2s$ and $C=C(n,s)$.
This bound holds, with the same constant $C$, for each ball $B_R(x_0)$, $x_0\in U$, where $R=d_{x_0}/2$.

\emph{Step 2.} Next we claim that if \eqref{dos} holds for each ball $B_{d_{x}/2}(x)$, $x\in U$, then \eqref{eq:2s-derivatives-more} holds. It is clear that this already yields
\begin{equation}\label{cotaambk}
 \sum_{l=1}^k d_x^{k-\alpha} \sup_{x\in U} |D^k u(x)|\le C\biggl( \|w\|_{C^{\alpha}(\R^n)}+ \|(-\Delta)^s w\|_{\beta;U}^{(2s-\alpha)}\biggr)
\end{equation}
where $k$ is the greatest integer less than $\beta+2s$.

To prove this claim we only have to control $[w]_{\beta+2s;U}^{(-\alpha)}$ ---see Definition \ref{definorm}. Let $\gamma\in(0,1)$ be such that $\beta+2s=k+\gamma$. We next bound
\[\frac{|D^k w(x)- D^k w(y)|}{|x-y|^{\gamma}}\]
when $d_x\geq d_y$ and $|x-y|\ge d_x/2$. This will yield the bound for $[w]_{\beta+2s;U}^{(-\alpha)}$, because if $|x-y|<d_x/2$ then $y\in B_{d_{x}/2}(x)$, and that case is done in Step 1.

We proceed differently in the cases $k=0$ and $k\ge 1$.
If $k=0$, then
\[d_x^{\beta+2s-\alpha}\frac{w(x)-w(y)}{|x-y|^{2s+\beta}}
=\left(\frac{d_x}{|x-y|}\right)^{\beta+2s-\alpha}\frac{w(x)-w(y)}{|x-y|^{\alpha}}\leq C\|w\|_{C^\alpha(\R^n)}.\]
If $k\ge 1$, then
\[d_x^{\beta+2s-\alpha}\frac{|D^k w(x)- D^k w(y)|}{|x-y|^{\gamma}}\le \biggl(\frac{d_x}{|x-y|}\biggr)^{\gamma}
d_x^{\beta+2s-\alpha-\gamma}|D^k w(x)- D^k w(y)| \le C \|w\|_{k;U}^{(-\alpha)}\,,\]
where we have used that $\beta+2s-\alpha-\gamma= k-\alpha$.

Finally, noting that for $x\in B_R(x_0)$ we have $R\leq d_{x_0}\leq 3R$, \eqref{eq:2s-derivatives-more} follows from \eqref{dos}, \eqref{cotaambk} and the definition of $\|w\|_{\alpha+2s;U}^{(-\alpha)}$ in \eqref{definorm}.
\end{proof}

Finally, to end this section, we prove Proposition \ref{prop:int-est-u}.

\begin{proof}[Proof of Proposition \ref{prop:int-est-u}]
Set $\alpha=s$ in Lemma \ref{refined-2s-gain}.
\end{proof}

\begin{rem}\label{remviscosity}
When $g$ is continuous, the notions of bounded weak solution and viscosity solution of \eqref{eqlin} ---and hence of \eqref{eqnonlin}--- coincide.

Indeed, let $u\in H^s(\R^n)$ be a weak solution of \eqref{eqlin}. Then, from Proposition \ref{prop:u-is-Cs} it follows that $u$ is continuous up to the boundary.
Let $u_{\varepsilon}$ and $g_{\varepsilon}$ be the standard regularizations of $u$ and $g$ by convolution with a mollifier. It is immediate to verify that, for $\varepsilon$ small enough, we have $(-\Delta)^s u_{\varepsilon}= g_\varepsilon$ in every subdomain $U\subset\subset\Omega$ in the classical sense.
Then, noting that $u_\varepsilon \to u$ and $g_\varepsilon \to g$ locally uniformly in $\Omega$, and applying the stability property for viscosity solutions \cite[Lemma 4.5]{CS}, we find that $u$ is a viscosity solution of \eqref{eqlin}.

Conversely, every viscosity solution of \eqref{eqlin} is a weak solution. This follows from three facts: the existence of weak solution, that this solution is a viscosity solution as shown before, and the uniqueness of viscosity solutions \cite[Theorem 5.2]{CS}.

As a consequence of this, if $g$ is continuous, any viscosity solution of \eqref{eqlin} belongs to $H^s(\R^n)$ ---since it is a weak solution.
This fact, which is not obvious, can also be proved without using the result on uniqueness of viscosity solutions.
Indeed, it follows from Proposition \ref{prop:int-est-u} and Lemma \ref{remlog}, which yield a stronger fact: that $(-\Delta)^{s/2}u\in L^p(\R^n)$ for all $p<\infty$.
Note that although we have proved Proposition \ref{prop:int-est-u} for weak solutions, its  proof is also valid ---with almost no changes--- for viscosity solutions.
\end{rem}

\section{Boundary regularity}
\label{sec4}

In this section we study the precise behavior near the boundary of the solution $u$ to problem \eqref{eqlin}, where $g\in L^\infty(\Omega)$.
More precisely, we prove that the function $u/\delta^s|_\Omega$ has a $C^{\alpha}(\overline\Omega)$ extension. This is stated in Theorem \ref{thm:v-is-Calpha}.

This result will be a consequence of the interior regularity results of Section \ref{sec3} and an oscillation lemma near the boundary, which can be seen as the nonlocal analog of Krylov's boundary Harnack principle; see Theorem 4.28 in \cite{Kazdan}.

The following proposition and lemma will be used to establish Theorem \ref{thm:v-is-Calpha}. They are proved in the Appendix.

\begin{prop}[1-D solution in half space, \cite{CRS}] \label{prop:solution}
The function $\varphi_0$, defined by
\begin{equation}
\varphi_0(x) =
\begin{cases}
0  \quad & \mbox{if } x\le 0\\
x^s & \mbox{if } x\ge 0\,,
\end{cases}
\end{equation}
satisfies $(-\Delta)^s\varphi_0 = 0$ in $\R_+$.
\end{prop}

The lemma below gives a subsolution in $B_1\setminus B_{1/4}$ whose support is $B_1\subset\R^n$ and such that it is comparable to $(1-|x|)^s$ in $B_1$.

\begin{lem}[Subsolution]\label{prop:subsolution}
There exist $C_2>0$ and a radial function $\varphi_2=\varphi_2(|x|)$ satisfying
\begin{equation}\label{eq:propsubsol}
\begin{cases}
(-\Delta)^s\varphi_2 \le 0  &\mbox{in }B_1\setminus B_{1/4}\\
\varphi_2 = 1 &\mbox{in }B_{1/4}\\
\varphi_2(x)\ge C_2(1-|x|)^s & \mbox{in } B_1\\
\varphi_2 = 0 \quad &\mbox{in }\R^n\setminus B_1 \,.
\end{cases}
\end{equation}
\end{lem}

To prove H\"older regularity of $u/\delta^s|_\Omega$ up to the boundary, we will control the oscillation of this function in sets near $\partial\Omega$ whose diameter goes to zero.
To do it, we will set up an iterative argument as it is done for second order equations.

Let us define the sets in which we want to control the oscillation and also auxiliary sets that are involved in the iteration.

\begin{defi}\label{defiDR}
Let $\kappa>0$ be a fixed small constant and let $\kappa'= 1/2+2\kappa$. We may take, for instance $\kappa= 1/16$, $\kappa' = 5/8$. Given a point $x_0$ in $\partial \Omega$ and $R>0$ let us define
\[ D_R = D_R(x_0) = B_R(x_0)\cap \Omega \]
and
\[D_{\kappa'R}^+ = D_{\kappa'R}^+(x_0)= B_{\kappa'R}(x_0)\cap\{x\in \Omega\,:\,-x\cdot \nu(x_0)\ge 2\kappa R\}\,,\]
where $\nu(x_0)$ is the unit outward normal at $x_0$; see Figure \ref{figura}. By $C^{1,1}$ regularity of the domain, there exists  $\rho_0>0$, depending on $\Omega$, such that the following inclusions hold for each $x_0\in \partial\Omega$ and $R\le \rho_0$:
\begin{equation}\label{eq:BkR(DR+)subsetDR}
B_{\kappa R}(y) \subset  D_R(x_0) \quad \mbox{for all }y\in D_{\kappa'R}^+(x_0)\,,
\end{equation}
and
\begin{equation}\label{eq:B4kR(y*+nu)subset}
B_{4\kappa R}(y^*-4\kappa R\nu(y^*))\subset D_R(x_0) \quad \mbox{and}\quad   B_{\kappa R}(y^*-4\kappa R\nu(y^*))\subset D_{\kappa'R}^+(x_0)\,
\end{equation}
for all $y \in D_{R/2}$, where $y^*\in\partial\Omega$ is the unique boundary point satisfying $|y-y^*|=\text{dist}(y,\partial\Omega)$.
Note that, since $R\leq\rho_0$, $y\in D_{R/2}$ is close enough to $\partial \Omega$ and hence the point $y^*-4\kappa R\nu(y^*)$ lays on the line joining $y$ and $y^*$; see Remark \ref{remrho0} below.
\end{defi}

\begin{figure}[htp]
\begin{center}
\includegraphics[]{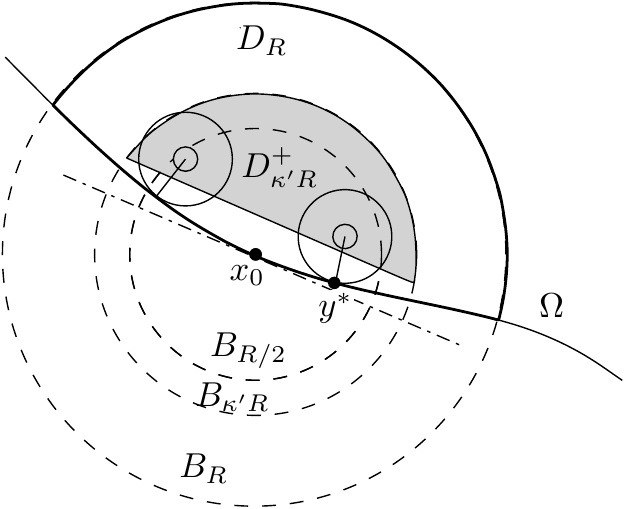}
\end{center}
\caption{\label{figura} The sets $D_R$ and $D_{\kappa'R}^+$}
\end{figure}

\begin{rem}\label{remrho0}
Throughout the paper, $\rho_0>0$ is a small constant depending only on $\Omega$, which we assume to be a bounded $C^{1,1}$ domain.
Namely, we assume that \eqref{eq:BkR(DR+)subsetDR} and \eqref{eq:B4kR(y*+nu)subset} hold whenever $R\le\rho_0$, for each $x_0\in\partial\Omega$, and also that every point on $\partial\Omega$ can be touched from both inside and outside $\Omega$ by balls of radius $\rho_0$.
In other words, given $x_0\in \partial\Omega$, there are balls of radius $\rho_0$, $B_{\rho_0}(x_1)\subset \Omega$ and $B_{\rho_0}(x_2)\subset\R^n\setminus \Omega$, such that $\overline{B_{\rho_0}(x_1)}\cap\overline{B_{\rho_0}(x_2)}=\{x_0\}$.
A useful observation is that all points $y$ in the segment that joins $x_1$ and $x_2$ ---through $x_0$--- satisfy
$\delta(y)= |y-x_0|$. Recall that $\delta={\rm dist}(\,\cdot\,,\partial\Omega)$.
\end{rem}

In the rest of this section, by $|(-\Delta)^su|\le K$ we mean that either $(-\Delta)^s u = g$ in the weak sense for some $g\in L^\infty$ satisfying $\|g\|_{L^\infty}\le K$ or that $u$ satisfies $-K \le (-\Delta)^s u \le K$ in the viscosity sense.

The first (and main) step towards Theorem \ref{thm:v-is-Calpha} is the following.

\begin{prop}\label{lem_main}
Let $\Omega$ be a bounded $C^{1,1}$ domain, and $u$ be such that $|(-\Delta)^s u|\le K$ in $\Omega$ and $u\equiv0$ in $\R^n \setminus \Omega$, for some constant $K$.
Given any $x_0\in \partial\Omega$, let $D_R$ be as in Definition \ref{defiDR}.

Then, there exist $\alpha\in(0,1)$ and $C$ depending only on $\Omega$ and $s$ ---but not on $x_0$--- such that
\begin{equation}\label{eq:lemmain}\sup_{D_R} u/\delta^s - \inf_{D_R} u/\delta^s \le C K R^\alpha\end{equation}
for all $R\leq\rho_0$, where $\rho_0>0$ is a constant depending only on $\Omega$.
\end{prop}

To prove Proposition \ref{lem_main} we need three preliminary lemmas.
We start with the first one, which might be seen as the fractional version of Lemma 4.31 in \cite{Kazdan}. Recall that $\kappa'\in(1/2,1)$ is a fixed constant throughout the section.
It may be useful to regard the following lemma as a bound by below for $\inf_{D_{R/2}} u/\delta^s$, rather than an upper bound for $\inf_{D_{\kappa'R}^+} u/\delta^s$.

\begin{lem}\label{lemA}
Let $\Omega$ be a bounded $C^{1,1}$ domain, and $u$ be such that $u\ge0$ in all of $\R^n$ and $|(-\Delta)^s u|\le K$ in $D_R$, for some constant $K$.
Then, there exists a positive constant $C$, depending only on $\Omega$ and $s$, such that
\begin{equation}\label{eq:lemA}
\inf_{D_{\kappa'R}^+} u/\delta^s \le C \bigl(\,\inf_{D_{R/2}} u/\delta^s + K R^s\bigr)
\end{equation}
for all $R\le \rho_0$, where $\rho_0>0$ is a constant depending only on $\Omega$.
\end{lem}

\begin{proof}
{\em Step 1.} We do first the case $K=0$. Let $R\leq \rho_0$, and let us call $m = \inf_{D_{\kappa'R}^+} u/\delta^s \ge 0$. We have $u\ge m \delta^s \ge m (\kappa R)^s$ on $D_{\kappa'R}^+$. The second inequality is a consequence of \eqref{eq:BkR(DR+)subsetDR}.

We scale the subsolution $\varphi_2$ in Lemma \ref{prop:subsolution} as follows, to use it as lower barrier:
\[\psi_R(x):= (\kappa R)^s \varphi_2\bigl(\textstyle \frac{x}{4\kappa R}\bigr)\,.\]
By \eqref{eq:propsubsol} we have
\[
\begin{cases}
(-\Delta)^s\psi_R \leq 0  &\mbox{in }B_{4\kappa R}\setminus B_{\kappa R}\\
\psi_R = (\kappa R)^s &\mbox{in }B_{\kappa R}\\
\psi_R \ge 4^{-s} C_2(4\kappa R-|x|)^s & \mbox{in } B_{4\kappa R}\setminus B_{\kappa R}\\
\psi_R \equiv 0 \quad &\mbox{in }\R^n\setminus B_{4\kappa R} \,.
\end{cases}
\]

Given $y \in D_{R/2}$, we have either $y\in D_{\kappa'R}^+$ or $\delta(y)<4\kappa R$, by \eqref{eq:B4kR(y*+nu)subset}.
If $y\in D_{\kappa'R}^+$ it follows from the definition of $m$ that $m \le u(y)/\delta(y)^s$.
If $\delta(y) < 4\kappa R$, let $y^*$ be the closest point to $y$ on $\partial \Omega$ and $\tilde y = y^* +4\kappa \nu(y^*)$. Again by \eqref{eq:B4kR(y*+nu)subset}, we have $B_{4\kappa R}(\tilde y)\subset D_R$ and $B_{\kappa R}(\tilde y)\subset D_{\kappa'R}^+$. But recall that $u\ge m (\kappa R)^s$ in $D_{\kappa'R}^+$, $(-\Delta)^s u= 0$ in $\Omega$, and $u\ge 0$ in $\R^n$.
Hence, $u(x)\ge m \psi_R(x-\tilde y)$ in all $\R^n$ and in particular
$u/\delta^s \ge 4^{-s}C_2 m$ on the segment joining $y^*$ and $\tilde y$, that contains $y$. Therefore,
\begin{equation}\label{eq:pflemA1}
\inf_{D_{\kappa'R}^+} u/\delta^s \le C \,\inf_{D_{R/2}} u/\delta^s\,.
\end{equation}

{\em Step 2.} If $K>0$ we consider $\tilde u$ to be the solution of
\[
\begin{cases}
(-\Delta)^s \tilde u = 0 \quad&\mbox{ in }D_R\\
\tilde u= u & \mbox{in }\R^n\setminus D_R.
\end{cases}
\]
By Step 1, \eqref{eq:pflemA1} holds with $u$ replaced by $\tilde u$.

On the other hand, $w=\tilde u - u$ satisfies $|(-\Delta)^s w|\le K$ and $w\equiv 0$ outside $D_R$. Recall that points of $\partial \Omega$ can be touched by exterior balls of radius less than $\rho_0$. Hence, using the rescaled supersolution $K R^{2s}\varphi_1(x/R)$ from Lemma \ref{prop:supersolution} as upper barrier and we readily prove, as in the proof
of Lemma \ref{lem-aux-Cs}, that
\[|w| \le C_1 K R^s\delta^s \quad \mbox{in }D_R\,.\]
Thus, \eqref{eq:lemA} follows.
\end{proof}

The second lemma towards Proposition \ref{lem_main}, which might be seen as the fractional version of Lemma 4.35 in \cite{Kazdan}, is the following.

\begin{lem}\label{lemB}
Let $\Omega$ be a bounded $C^{1,1}$ domain, and $u$ be such that $u\ge0$ in all of $\R^n$ and $|(-\Delta)^s u|\le K$ in $D_R$, for some constant $K$.
Then, there exists a positive constant $C$, depending on $\Omega$ and $s$, such that
\begin{equation}\label{eq:lemB}
\sup_{D_{\kappa'R}^+} u/\delta^s \le C \bigl(\,\inf_{D_{\kappa'R}^+} u/\delta^s + K R^s\bigr)
\end{equation}
for all $R\le\rho_0$, where $\rho_0>0$ is a constant depending only on $\Omega$.
\end{lem}

\begin{proof}
{\em Step 1.} Consider first the case $K=0$. In this case \eqref{eq:lemB} follows from the Harnack inequality for the fractional Laplacian \cite{L} ---note that we assume $u\ge 0$ in all $\R^n$. Indeed, by \eqref{eq:BkR(DR+)subsetDR}, for each $y\in D_{\kappa'R}^+$ we have $B_{\kappa R}(y)\subset D_R$ and hence  $(-\Delta)^s u= 0$ in $B_{\kappa R}(y)$. Then we may cover $D_{\kappa'R}^+$ by a finite number of balls $B_{\kappa R/2}(y_i)$, using the same (scaled) covering for all $R\leq \rho_0$, to obtain
\[\sup_{B_{\kappa R/2}(y_i)} u\leq C\inf_{B_{\kappa R/2}(y_i)} u.\]
Then, \eqref{eq:lemB} follows since $(\kappa R/2)^s\leq \delta^s\leq (3\kappa R/2)^s$ in $B_{\kappa R/2}(y_i)$ by \eqref{eq:BkR(DR+)subsetDR}.

{\em Step 2.}  When $K>0$, we prove \eqref{eq:lemB} by using a similar argument as in Step 2 in the proof of Proposition \ref{lemA}.
\end{proof}

Before proving Lemma \ref{lapsdeltas} we give an extension lemma ---see \cite[Theorem 1, Section 3.1]{EG} where the case $\alpha=1$ is proven in full detail.

\begin{lem} \label{prop:extension-op-E}
Let $\alpha\in(0,1]$ and $V\subset\R^n$ a bounded domain. There exists a (nonlinear) map $E:C^{0,\alpha}(\overline V)\rightarrow C^{0,\alpha}(\R^n)$ satisfying
\[ E(w)\equiv w \quad \mbox{in }\overline V,\ \ \ [E(w)]_{C^{0,\alpha}(\R^n)}\le [w]_{C^{0,\alpha}(\overline V)},\ \ \ \mbox{and}\ \ \ \|E(w)\|_{L^\infty(\R^n)}\le \|w\|_{L^\infty(V)}\]
for all $w\in C^{0,\alpha}(\overline V)$.
\end{lem}

\begin{proof}
It is immediate to check that
\[E(w)(x)=\min\left\{\min_{z\in \overline V}\left\{w(z)+
[w]_{C^{\alpha}(\overline V)}|z-x|^\alpha\right\},\|w\|_{L^\infty(V)}\right\}\]
satisfies the conditions since, for all $x,y,z$ in $\R^n$,
\[|z-x|^\alpha \le |z-y|^\alpha+|y-x|^\alpha\,.\]
\end{proof}

We can now give the third lemma towards Proposition \ref{lem_main}. This lemma, which is related to Proposition \ref{prop:solution}, is crucial. It states that $\delta^s|_\Omega$, extended by zero outside $\Omega$, is an approximate solution in a neighborhood of $\partial\Omega$ inside $\Omega$.

\begin{lem}\label{lapsdeltas}
Let $\Omega$ be a bounded $C^{1,1}$ domain, and $\delta_0=\delta\chi_\Omega$ be the distance function in $\Omega$ extended by zero outside $\Omega$. Let $\alpha=\min\{s,1-s\}$, and $\rho_0$ be given by Remark \ref{remrho0}. Then,
\[(-\Delta)^s \delta_0^s\qquad \textrm{belongs to }\ C^{\alpha}(\overline{\Omega_{\rho_0}})\,,\]
where $\Omega_{\rho_0}=\Omega\cap\{\delta<\rho_0\}$.
In particular,
\[|(-\Delta)^s \delta_0^s| \le C_\Omega\quad \mbox{in } \Omega_{\rho_0}\,,\]
where $C_\Omega$ is a constant depending only on $\Omega$ and $s$.
\end{lem}

\begin{proof}
Fix a point $x_0$ on $\partial\Omega$ and denote, for $\rho>0$, $B_\rho= B_{\rho}(x_0)$.
Instead of proving that
\[(-\Delta)^s\delta_0^s=c_{n,s}{\rm PV} \int_{\R^n}\frac{\delta_0(x)^s-\delta_0(y)^s}{|x-y|^{n+2s}}dy\]
is $C^{\alpha}(\overline{\Omega\cap B_{\rho_0}})$ ---as a function of $x$---, we may equivalently prove that
\begin{equation}\label{prove1}{\rm PV} \int_{B_{2\rho_0}}\frac{\delta_0(x)^s-\delta_0(y)^s}{|x-y|^{n+2s}}dy\qquad \mbox{belongs to}\qquad C^{\alpha}(\overline{\Omega\cap B_{\rho_0}}).
\end{equation}
This is because the difference
\[\frac{1}{c_{n,s}}(-\Delta)^s\delta_0^s-{\rm PV} \int_{B_{2\rho_0}}\frac{\delta_0(x)^s-\delta_0(y)^s}{|x-y|^{n+2s}}dy=\int_{\R^n \setminus B_{2\rho_0}}\frac{\delta_0(x)^s-\delta_0(y)^s}{|x-y|^{n+2s}}dy\,\]
belongs to $C^{s}(\overline{B_{\rho_0}})$, since $\delta_0^s$ is $C^{s}(\R^n)$ and $|x|^{-n-2s}$ is integrable and smooth outside a neighborhood of $0$.

To see \eqref{prove1}, we flatten the boundary.
Namely, consider a $C^{1,1}$ change of variables $X=\Psi(x)$, where $\Psi: B_{3\rho_0} \rightarrow V\subset \R^n$ is a $C^{1,1}$ diffeomorphism, satisfying that $\partial\Omega$ is mapped onto $\{X_n=0\}$, $\Omega\cap B_{3\rho_0}$ is mapped into $\R^n_+$, and $\delta_0(x)= (X_n)_+$.
Such diffeomorphism exists because we assume $\Omega$ to be $C^{1,1}$.
Let us respectively call $V_1$ and $V_2$ the images of $B_{\rho_0}$ and $B_{2\rho_0}$ under $\Psi$.
Let us denote the points of $V\times V$ by $(X,Y)$. We consider the functions $x$ and $y$, defined in $V$, by  $x= \Psi^{-1}(X)$ and $y= \Psi^{-1}(Y)$.
With these notations, we have
\[x-y=-D\Psi^{-1}(X)(X-Y)+\mathcal{O}\left(|X-Y|^2\right),\]
and therefore
\begin{equation}\label{x-y}|x-y|^2=(X-Y)^TA(X)(X-Y)+\mathcal{O}\left(|X-Y|^3\right),\end{equation}
where
\[A(X)=\left(D\Psi^{-1}(X)\right)^TD\Psi^{-1}(X)\]
is a symmetric matrix, uniformly positive definite in $\overline{V_2}$.
Hence,
\[{\rm PV} \int_{B_{2\rho_0}}\frac{\delta_0(x)^s-\delta_0(y)^s}{|x-y|^{n+2s}}dy=
{\rm PV} \int_{V_2}\frac{(X_n)_+^s-(Y_n)_+^s}{\left|(X-Y)^TA(X)(X-Y)\right|^{\frac{n+2s}{2}}}g(X,Y)dY,\]
where we have denoted
\[g(X,Y)=\left(\frac{(X-Y)^TA(X)(X-Y)}{|x-y|^2}\right)^{\frac{n+2s}{2}}J(Y)\]
and $J=|\det D\Psi^{-1}|$.
Note that we have $g\in C^{0,1}(\overline{V_2\times V_2})$, since $\Psi$ is $C^{1,1}$ and we have \eqref{x-y}.

Now we are reduced to proving that
\begin{equation}\label{prove2}
\psi_1(X):= {\rm PV}\int_{V_2}\frac{(X_n)_+^s-(Y_n)_+^s}{\left|(X-Y)^TA(X)(X-Y)\right|^{\frac{n+2s}{2}}}\,g(X,Y)dY,
\end{equation}
belongs to $C^{\alpha}(\overline{V_1^+})$ (as a function of $X$),
where $V_1^+=V_1\cap \{X_n>0\}$.

To prove this, we extend the Lipschitz function $g\in C^{0,1}(\overline{V_2\times V_2})$ to all $\R^n$.
Namely, consider the function $g^*=E(g)\in C^{0,1}(\R^n\times\R^n)$ provided by Proposition \ref{prop:extension-op-E}, which satisfies
\[ g^*\equiv g \mbox{ in }\overline{V_2\times V_2} \quad \mbox{and}\quad \| g^*\|_{C^{0,1}(\R^n\times\R^n)}\le\|g\|_{C^{0,1}(\overline{V_2\times V_2})}\,.\]

By the same argument as above, using that $V_1\subset\subset V_2$, we have that $\psi_1\in C^{\alpha}(\overline{V_1^+})$ if and only if so is the function
\[\psi(X)= {\rm PV}\int_{\R^n}\frac{(X_n)_+^s-(Y_n)_+^s}{\left|(X-Y)^TA(X)(X-Y)\right|^{\frac{n+2s}{2}}}\,g^*(X,Y)dY.\]

Furthermore, from $g^*$ define $\tilde g \in C^{0,1}(\overline{V_2}\times\R^n)$ by $\tilde g (X,Z) = g^*( X,X+MZ) \det M$, where $M=M(X)= D\Psi(X)$.
Then, using the change of variables $Y=X+MZ$ we deduce
\[\psi(X)= {\rm PV}\int_{\R^n}\frac{(X_n)_+^s-\bigl(e_n\cdot (X+ MZ)\bigr)_+^s}{|Z|^{n+2s}}\,\tilde g(X,Z)dZ.\]

Next, we prove that $\psi\in C^\alpha(\R^n)$, which concludes the proof.
Indeed, taking into account that the function $(X_n)_+^s$ is $s$-harmonic in $\R^n_+$ ---by Proposition \ref{prop:solution}---
we obtain
\[{\rm PV}\int_{\R^n}\frac{(e'\cdot X')_+^s-(e'\cdot(X'+Z))_+^s}{|Z|^{n+2s}}dZ=0 \]
for every $e'\in \R^n$ and for every $X'$ such that $e'\cdot X'>0$.
Thus, letting $e'= e_n^T M$ and $X'=M^{-1}X$ we deduce
\[{\rm PV}\int_{\R^n}\frac{(X_n)_+^s-\bigl(e_n\cdot (X+ MZ)\bigr)_+^s}{|Z|^{n+2s}}dZ=0\]
for every $X$ such that $(e_n^T M)\cdot(M^{-1}X)>0$, that is, for every $X\in\R^n_+$.

Therefore, it holds
\[\psi(X)=\int_{\R^n}\frac{\phi(X,0)-\phi(X,Z)}{|Z|^{n+2s}}\bigl(\tilde g(X,Z)- \tilde g(X,0)\bigr)dZ,\]
where
\[\phi(X,Z)=(e_n\cdot(X+MZ))_+^s\]
satisfies $[\phi]_{C^s(\overline{V_2}\times \R^n)}\leq C$, and $\|\tilde g\|_{C^{0,1}(\overline{V_2}\times \R^n)}\leq C$.

Let us finally prove that $\psi$ belongs to $C^\alpha(\overline{V_1^+})$. To do it, let $X$ and $\bar X$ be in $\overline{V_1^+}$. Then, we have
\[\psi(X)-\psi(\bar X)=\int_{\R^n} \frac{\Theta(X,\bar X,Z)}{|Z|^{n+2s}}dZ,\]
where
\begin{equation}
\begin{split}
\Theta(X,&\bar X,Z)= \bigl(\phi(X,0)-\phi(X,Z)\bigr)\bigl(\tilde g(X,Z)- \tilde g(X,0)\bigr)\\
&\hspace{15mm} -\bigl(\phi(\bar X,0)-\phi(\bar X,Z)\bigr)\bigl(\tilde g(\bar X,Z)-  \tilde g(\bar X,0)\bigr)\\
&= \bigl(\phi(X,0)-\phi(X,Z)- \phi(\bar X,0)+\phi(\bar X,Z)\bigr)\bigl(\tilde g(X,Z)- \tilde g(X,0)\bigr)
\\& \quad -\bigl(\phi(\bar X,0)-\phi(\bar X,Z)\bigr)\bigl(\tilde g(X,Z)-  \tilde g(X,0) -  \tilde g(\bar X,Z)+ \tilde g(\bar X,0)\bigr).
\end{split}
\end{equation}

Now, on the one hand, it holds
\begin{equation}\label{theta1}|\Theta(X,\bar X,Z)|\leq C|Z|^{1+s},\end{equation}
since $[\phi]_{C^s(\overline{V_2}\times \R^n)}\leq C$ and $\|\tilde g\|_{C^{0,1}(\overline{V_2}\times \R^n)}\leq C$.

On the other hand, it also holds
\begin{equation}\label{theta2}|\Theta(X,\bar X,Z)|\leq C|X-\bar X|^s \min\{|Z|,|Z|^s\}.\end{equation}
Indeed, we only need to observe that
\[\begin{split}
\left|\tilde g(X,Z)-  \tilde g(X,0) -  \tilde g(\bar X,Z)+ \tilde g(\bar X,0)\right| &\leq C\min\bigl\{\min\{|Z|,1\}, |X-\bar X|\bigr\}\\
&\le C \min \{|Z|^{1-s},1\}|X-\bar X|^s.
\end{split}
\]

Thus, letting $r=|X-\bar X|$ and using  \eqref{theta1} and \eqref{theta2}, we obtain
\[
\begin{split}
|\psi(X)-\psi(\bar X)|&\le \int_{\R^n}\frac{|\Theta(X,\bar X,Z)|}{|Z|^{n+2s}}dZ
\\
&\le \int_{B_r} \frac{C|Z|^{1+s}}{|Z|^{n+2s}}dZ + \int_{\R^n\setminus B_r} \frac{C r^s \min\{|Z|,|Z|^s\}}{|Z|^{n+2s}}dZ
\\ &\le C r^{1-s}+C\max\{r^{1-s},r^s\}\,,
\end{split}
\]
as desired.
\end{proof}

Next we prove Proposition \ref{lem_main}.

\begin{proof}[Proof of Proposition \ref{lem_main}]
By considering $u/K$ instead of $u$ we may assume that $K=1$, that is, that $|(-\Delta)^s u|\le 1$ in $\Omega$.
Then, by Claim \ref{Linftybound} we have $\|u\|_{L^\infty(\R^n)}\le C$ for some constant $C$ depending only on $\Omega$ and $s$.

Let $\rho_0>0$ be given by Remark \ref{remrho0}.
Fix $x_0\in \partial \Omega$.
We will prove that there exist constants $C_0>0$, $\rho_1\in(0,\rho_0)$, and $\alpha\in(0,1)$, depending only on $\Omega$ and $s$, and monotone sequences $(m_k)$ and $(M_k)$ such that, for all $k\geq0$,
\begin{equation}\label{eq:prooflem1}
 M_k - m_k =  4^{-\alpha k}\,,\quad -1\le m_k\le m_{k+1}< M_{k+1}\le M_k\le 1\,,
\end{equation}
and
\begin{equation}\label{eq:prooflem2}
 m_k \le C_0^{-1}u/\delta^s \le M_k \quad \mbox{in } D_{R_k}= D_{R_k}(x_0)\,, \quad \mbox{where } R_k = \rho_1 4^{-k}.
\end{equation}
Note that \eqref{eq:prooflem2} is equivalent to the following inequality in $B_{R_k}$ instead of $D_{R_k}$ --- recall that $D_{R_k}= B_{R_k}\cap\Omega$.
\begin{equation}\label{eq:prooflem3}
 m_k \delta_0^s \le C_0^{-1}u \le M_k \delta_0^s \quad \mbox{in } B_{R_k}= B_{R_k}(x_0)\,, \quad \mbox{where } R_k = \rho_1 4^{-k}\,.
\end{equation}

If there exist such sequences, then \eqref{eq:lemmain} holds for all $R\leq \rho_1$ with $C=4^\alpha C_0/\rho_1^\alpha$. 
Then, by increasing the constant $C$ if necessary, \eqref{eq:lemmain} holds also for every $R\leq \rho_0$.

Next we construct $\{M_k\}$ and $\{m_k\}$ by induction.

By Lemma \ref{lem-aux-Cs}, we find that there exist $m_0$ and $M_0$ such that \eqref{eq:prooflem1}  and \eqref{eq:prooflem2} hold for $k=0$ provided we pick $C_0$ large enough depending on $\Omega$ and $s$.

Assume that we have sequences up to $m_k$ and $M_k$. We want to prove that there exist $m_{k+1}$ and $M_{k+1}$ which fulfill the requirements. Let
\begin{equation}\label{123}
u_k = C_0^{-1}u - m_k \delta_0^s\,.
\end{equation}

We will consider the positive part $u_{k}^+$ of $u_k$ in order to have a nonnegative function in all of $\R^n$ to which we can apply Lemmas \ref{lemA} and \ref{lemB}.
Let $u_{k}= u_k^+-u_k^-$. Observe that, by induction hypothesis,
\begin{equation}\label{321}
u_k^+ = u_k  \quad\mbox{and}\quad u_k^-= 0 \quad \mbox{in }B_{R_k}\,.
\end{equation}
Moreover, $C_0^{-1}u \ge m_j\delta_0^s$ in $B_{R_j}$ for each $j\le k$. Therefore, by \eqref{123} we have
\[u_k \ge (m_{j}-m_k) \delta_0^s \ge (m_{j}-M_{j}+M_k-m_k) \delta_0^s \ge (-4^{-\alpha j}  +4^{-\alpha k}) \delta_0^s \quad \mbox{in } B_{R_j}.\]
But clearly $0\le \delta_0^s\le R_{j}^s = \rho_1^s 4^{-js}$ in $B_{R_j}$, and therefore using $R_j=\rho_1 4^{-j}$
\[u_k \ge - \rho_1^{-\alpha}  R_j^s(R_j^\alpha-R_k^\alpha) \quad \mbox{in }\ B_{R_j}\ \mbox{for each}\ \ j\le k\,.\]

Thus, since for every $x\in B_{R_0}\setminus B_{R_k}$ there is $j<k$ such that
\[ |x-x_0|< R_j = \rho_1 4^{-j} \le 4|x-x_0|,\]
we find
\begin{equation}\label{eq:pflem3}
u_{k}(x)\ge - \rho_1^{-\alpha} R_k^{\alpha+s}  \biggl|\frac{4(x-x_0)}{R_k}\biggr|^s \biggl(\biggl|\frac{4(x-x_0)}{R_k}\biggr|^\alpha - 1\biggr) \quad \mbox{outside } B_{R_k}\,.
\end{equation}

By \eqref{eq:pflem3} and \eqref{321}, at $x\in B_{R_k/2}(x_0)$ we have
\[
\begin{split}
0\le -(-\Delta)^s u_{k}^- (x) &= c_{n,s}\int_{x+y\notin B_{R_k}}\frac{u_k^-(x+y)}{|y|^{n+2s}}\,dy\\
&\le c_{n,s}\,\rho_1^{-\alpha}\int_{|y|\ge R_k/2}  R_k^{\alpha+s} \biggl|\frac{8y}{R_k}\biggr|^s \biggl(\biggl|\frac{8y}{R_k}\biggr|^\alpha - 1\biggr) |y|^{-n-2s}\,dy \\
&= C\rho_1^{-\alpha} R_k^{\alpha-s}\int_{|z|\ge 1/2} \frac{|8z|^{s}(|8z|^\alpha-1)}{|z|^{n+2s}}\,dz\\
&\le \varepsilon_0\rho_1^{-\alpha} R_k^{\alpha-s},
\end{split}
\]
where $\varepsilon_0= \varepsilon_0(\alpha)\downarrow 0$ as $\alpha\downarrow 0$ since $|8z|^\alpha\rightarrow 1$.

Therefore, writing $u_k^+=C_0^{-1}u-m_k\delta_0^s+u_k^-$ and using Lemma \ref{lapsdeltas}, we have
\[\begin{split}
|(-\Delta)^s u_{k}^+| &\le C_0^{-1}|(-\Delta)^s u|+ m_k|(-\Delta)^s \delta_0^s| + |(-\Delta)^s (u_k^-)|\\
&\le  (C_0^{-1} + C_\Omega) + \varepsilon_0\rho_1^{-\alpha} R_k^{\alpha -s}\\
&\le \textstyle \bigl(C_1\rho_1^{s-\alpha}+\varepsilon_0\rho_1^{-\alpha}\bigr) R_k^{\alpha-s}\qquad\textrm{ in }D_{R_k/2}.
\end{split}\]
In the last inequality we have just used $R_k\leq\rho_1$ and $\alpha\leq s$.

Now we can apply Lemmas \ref{lemA} and \ref{lemB} with $u$ in its statements replaced by $u_{k}^+$, recalling that
\[\textstyle u_{k}^+ = u_k =  C_0^{-1}u- m_k \delta^s \quad \mbox{in }D_{R_k}\]
to obtain
\begin{eqnarray}
\sup_{D_{\kappa'R_k/2}^+} (C_0^{-1}u/\delta^s-m_k) &\le C \nonumber \biggl(\inf_{D_{\kappa'R_k/2}^+} (C_0^{-1}u/\delta^s-m_k) + \bigl(C_1\rho_1^{s-\alpha} +\varepsilon_0\rho_1^{-\alpha}\bigr) R_k^\alpha \biggr)\\
&\le C \biggl(\inf_{D_{R_k/4}} (C_0^{-1}u/\delta^s-m_k)+ \bigl(C_1\rho_1^{s-\alpha} +\varepsilon_0\rho_1^{-\alpha}\bigr) R_k^\alpha \biggr)\,.\label{eq:pflema1}
\end{eqnarray}

Next we can repeat all the argument ``upside down'', that is, with the functions $u^k = M_k \delta^s - u$ instead of $u_k$. In this way we obtain, instead of \eqref{eq:pflema1}, the following:
\begin{equation}\label{eq:pflema2}
\sup_{D_{\kappa'R_k/2}^+} (M_k - C_0^{-1}u/\delta^s) \le C \biggl(\inf_{D_{R_k/4}} (M^k-C_0^{-1}u/\delta^s)+ \bigl(C_1\rho_1^{s-\alpha} +\varepsilon_0\rho_1^{-\alpha}\bigr) R_k^\alpha \biggr).
\end{equation}

Adding \eqref{eq:pflema1} and \eqref{eq:pflema2} we obtain
\begin{equation}\label{eq:pflema3}
\begin{split}
M_k-m_k &\le C \biggl(\inf_{D_{R_k/4}} (C_0^{-1}u/\delta^s-m_k) + \inf_{D_{R_k/4}} (M_k-C_0^{-1}u/\delta^s) + \bigl(C_1\rho_1^{s-\alpha} +\varepsilon_0\rho_1^{-\alpha}\bigr) R_k^\alpha \biggr)\\
&= C\biggl(\inf_{D_{R_{k+1}}} C_0^{-1}u/\delta^s - \sup_{D_{R_{k+1}}} C_0^{-1}u/\delta^s +M_k-m_k+ \bigl(C_1\rho_1^{s-\alpha} +\varepsilon_0\rho_1^{-\alpha}\bigr) R_k^\alpha \biggr),
\end{split}
\end{equation}
and thus, using that $M_k-m_k = 4^{-\alpha k}$ and $R_k = \rho_1 4^{-k}$,
\[\sup_{D_{R_{k+1}}} C_0^{-1}u/\delta^s - \inf_{D_{R_{k+1}}}  C_0^{-1}u/\delta^s \le \bigl(\textstyle \frac{C-1}{C} +C_1\rho_1^s +\varepsilon_0\bigr) 4^{-\alpha k}\,.\]

Now we choose $\alpha$ and $\rho_1$ small enough so that
\[\frac{C-1}{C} +C_1\rho_1^s +\varepsilon_0(\alpha) \le 4^{-\alpha}.\]
This is possible since $\varepsilon_0(\alpha)\downarrow 0$ as $\alpha\downarrow 0$ and the constants $C$ and $C_1$ do not depend on $\alpha$ nor $\rho_1$ ---they depend only on $\Omega$ and $s$.
Then, we find
\[\sup_{D_{R_{k+1}}} C_0^{-1}u/\delta^s - \inf_{D_{R_{k+1}}}  C_0^{-1}u/\delta^s \le  4^{-\alpha (k+1)},\]
and thus we are able to choose $m_{k+1}$ and $M_{k+1}$ satisfying \eqref{eq:prooflem1} and \eqref{eq:prooflem2}.
\end{proof}

Finally, we give the:

\begin{proof}[Proof of Theorem \ref{thm:v-is-Calpha}]
Define $v=u/\delta^s|_\Omega$ and $K=\|g\|_{L^\infty(\Omega)}$. 
As in the proof of Proposition \ref{lem_main}, by considering $u/K$ instead of $u$ we may assume that $|(-\Delta)^s u|\leq 1$ in $\Omega$ and that $\|u\|_{L^{\infty}(\Omega)}\leq C$ for some constant $C$ depending only on $\Omega$ and $s$.

First we claim that there exist constants $C$, $M>0$, $\widetilde\alpha\in (0,1)$ and $\beta\in(0,1)$, depending only on $\Omega$ and $s$, such that
\begin{itemize}
\item[(i)] $\|v\|_{L^\infty(\Omega)}\le C$.
\item[(ii)] For all $x\in \Omega$, it holds the seminorm bound
\[   [v]_{C^{\beta}(\overline{B_{R/2}(x)})} \le C \left(1+R^{-M}\right),\]
where $R={\rm dist}(x,\mathbb R^n\setminus\Omega)$.
\item[(iii)] For each $x_0\in \partial \Omega$ and for all $\rho>0$ it holds
\[   \sup_{B_{\rho}(x_0)\cap \Omega} v-\inf_{B_{\rho}(x_0)\cap \Omega} v \leq C {\rho}^{\widetilde \alpha}.\]
\end{itemize}

Indeed, it follows from Lemma \ref{lem-aux-Cs} that $\|v\|_{L^\infty(\Omega)}\le C$ for some $C$ depending only on $\Omega$ and $s$.
Hence, (i) is satisfied.

Moreover, if $\beta\in(0,2s)$, it follows from Lemma \ref{lem-sarp-Cs-bounds-u} that for every $x\in \Omega$,
\[
[u]_{C^{\beta}(B_{R/2}(x))} \le CR^{-\beta},\qquad \beta\in(0,2s),
\]
where $R=\delta(x)$.
But since $\Omega$ is $C^{1,1}$, then provided $\delta(x)<\rho_0$ we will have
\[\|\delta^{-s}\|_{L^{\infty}(B_{R/2}(x))}\le CR^{-s}\quad \mbox{and}\quad [\delta^{-s}]_{C^{0,1}(B_{R/2}(x))}\le CR^{-s-1}\]
and hence, by interpolation,
\[[\delta^{-s}]_{C^{\beta}(B_{R/2}(x))}\le CR^{-s-\beta}\]
for each $\beta\in(0,1)$. 
Thus, since $v= u\delta^{-s}$, we find
\[[v]_{C^{\beta}(B_{R/2}(x))} \le C\left(1+R^{-s-\beta}\right)\]
for all $x\in \Omega$ and $\beta<\min\{1,2s\}$. Therefore hypothesis (ii) is satisfied.
The constants $C$ depend only on $\Omega$ and $s$.

In addition, using Proposition \ref{lem_main} and that $\|v\|_{L^\infty(\Omega)}\leq C$, we deduce that hypothesis (iii) is satisfied.

Now, we claim that (i)-(ii)-(iii) lead to
\[  [v]_{C^{\alpha}(\overline \Omega)} \le C,\]
for some ${\alpha}\in(0,1)$ depending only on $\Omega$ and $s$.

Indeed, let $x, y\in \Omega$, $R={\rm dist}(x,\R^n\setminus\Omega) \geq {\rm dist}(y,\R^n\setminus\Omega)$, and $r=|x-y|$.
Let us see that
$|v(x)-v(y)|\leq Cr^{\alpha}$ for some ${\alpha}>0$.

If $r\geq1$ then it follows from (i). Assume $r<1$, and let $p\geq1$ to be chosen later.
Then, we have the following dichotomy:

{\em Case 1.} Assume $r\ge R^p/2$. Let $x_0, y_0 \in \partial \Omega$ be such that $|x-x_0|={\rm dist}(x,\R^n\setminus\Omega)$ and $|y-y_0|={\rm dist}(y,\R^n\setminus\Omega)$.
Then, using (iii) and the definition of $R$ we deduce
\[ |v(x)-v(y)| \le |v(x)-v(x_0)|+|v(x_0)-v(y_0)|+|v(y_0)-v(y)|  \le  C  R^{\widetilde\alpha} \le Cr^{\widetilde\alpha/p}.\]

{\em Case 2.} Assume $r\le R^p/2$.
Hence, since $p\ge 1$, we have $y\in B_{R/2}(x)$. Then, using (ii) we obtain
\[ |v(x)-v(y)| \le C (1+R^{-M}) r^\beta \le  C\left( 1+ r^{-M/p}\right) r^{\beta} \le C r^{\beta- M/p}.\]

To finish the proof we only need to choose $p>M/\beta$ and take ${\alpha}= \min\{\widetilde\alpha/p, \beta -M/p \}$.
\end{proof}

\section{Interior estimates for $u/\delta^s$}
\label{sec5}

The main goal of this section is to prove the $C^\gamma$ bounds in $\Omega$ for the function $u/\delta^s$ in Theorem \ref{thm:int-est-v}.

To prove this result we find an equation for the function $v=u/\delta^s|_\Omega$, that is derived below.
This equation is nonlocal, and thus, we need to give values to $v$ in $\R^n\setminus\Omega$, although we want an equation only in $\Omega$.
It might seem natural to consider $u/\delta^s$, which vanishes outside $\Omega$ since $u\equiv0$ there, as an extension of $u/\delta^s|_\Omega$. However, such extension is discontinuous through $\partial\Omega$, and it would lead to some difficulties.

Instead, we consider a $C^\alpha(\R^n)$ extension of the function $u/\delta^s|_\Omega$, which is $C^\alpha(\overline\Omega)$ by Theorem \ref{thm:v-is-Calpha}. Namely, throughout this section, let $v$ be the $C^\alpha(\R^n)$ extension of $u/\delta^s|_\Omega$ given by Lemma \ref{prop:extension-op-E}.

Let $\delta_0=\delta\chi_\Omega$, and note that $u=v\delta_0^s$ in $\R^n$.
Then, using \eqref{eqlin} we have
\[g(x) = (-\Delta)^s(v\delta_0^s) = v(-\Delta)^s \delta_0^s+ \delta_0^s(-\Delta)^s v -I_s(v,\delta_0^s)\,\]
in $\Omega_{\rho_0}=\{x\in\Omega\,:\,\delta(x)<\rho_0\}$, where
\begin{equation}\label{Is} I_s(w_1,w_2)(x)= c_{n,s}\int_{\R^n} \frac{\bigl(w_1(x)-w_1(y)\bigr)\bigl(w_2(x)-w_2(y)\bigr)}{|x-y|^{n+2s}}\,dy
\end{equation}
and $\rho_0$ is a small constant depending on the domain; see Remark \ref{remrho0}.
Here, we have used that $(-\Delta)^s(w_1w_2)=w_1(-\Delta)^sw_2+w_2(-\Delta)^sw_1-I_s(w_1,w_2)$, which follows easily from \eqref{laps}. This equation is satisfied pointwise in $\Omega_{\rho_0}$, since $g$ is $C^\alpha$ in $\Omega$.
We have to consider $\Omega_{\rho_0}$ instead of $\Omega$  because the distance function is $C^{1,1}$ there and thus we can compute $(-\Delta)^s\delta_0^s$. In all $\Omega$ the distance function $\delta$ is only Lipschitz and hence $(-\Delta)^s\delta_0^s$ is singular for $s\ge\frac12$.

Thus, the following is the equation for $v$:
\begin{equation}\label{equaciov}
(-\Delta)^sv= \frac{1}{\delta_0^s}\biggl( g(x) - v (-\Delta)^s \delta_0^s + I_s(v,\delta_0^s)\biggr) \quad\mbox{in } \Omega_{\rho_0}\,.
\end{equation}

From this equation we will obtain the interior estimates for $v$. More precisely, we will obtain a priori bounds for the interior H\"older norms of $v$, treating
$\delta_0^{-s}I_s(v,\delta_0^s)$ as a lower order term.
For this, we consider the weighted H\"older norms given by Definition \ref{definorm}.

Recall that, in all the paper, we denote $C^\beta$ the space $C^{k,\beta'}$, where $\beta=k+\beta'$ with $k$ integer and $\beta'\in(0,1]$.

In Theorem \ref{thm:v-is-Calpha} we have proved that $u/\delta^s|_{\Omega}$ is $C^{\alpha}(\overline\Omega)$ for some $\alpha\in(0,1)$, with an estimate.
From this $C^\alpha$ estimate and from the equation for $v$ \eqref{equaciov}, we will find next the estimate for $\|u/\delta^s\|_{\gamma;\Omega}^{(-\alpha)}$ stated in  Theorem \ref{thm:int-est-v}.

The proof of this result relies on some preliminary results below.

Next lemma is used to control the lower order term $\delta_0^{-s}I_s(v,\delta_0^s)$ in the equation \eqref{equaciov} for $v$.

\begin{lem} \label{lem:bound-I} Let $\Omega$ be a bounded $C^{1,1}$ domain, and $U\subset \Omega_{\rho_0}$ be an open set. Let $s$ and $\alpha$ belong to $(0,1)$ and satisfy $\alpha+s\le 1$ and $\alpha< s$. Then,
\begin{equation}\label{eq:bound-I}
\|I_s(w,\delta_0^s)\|_{\alpha;U}^{(s-\alpha)}\le C\biggl( [w]_{C^{\alpha}(\R^n)}+ [w]_{\alpha+s;U}^{(-\alpha)}\biggr)\,,
\end{equation}
for all $w$ with finite right hand side.
The constant $C$ depends only on $\Omega$, $s$, and $\alpha$.
\end{lem}

To prove Lemma \ref{lem:bound-I} we need the next

\begin{lem}\label{claim:I}
 Let $U\subset \R^n$ be a bounded open set. Let $\alpha_1,\alpha_2,\in(0,1)$ and $\beta\in(0,1]$ satisfy $\alpha_i < \beta$ for $i=1,2$, $\alpha_1+\alpha_2<2s$, and $s<\beta<2s$.
Assume that  $w_1,w_2\in C^\beta(U)$. Then,
\begin{equation}\label{eq:bound-Iclaim}
\|I_s(w_1,w_2)\|_{2\beta-2s;U}^{(2s-\alpha_1-\alpha_2)}\le C\left( [w_1]_{C^{\alpha_1}(\R^n)}+ [w_1]_{\beta;U}^{(-\alpha_1)}\right)\left( [w_2]_{C^{\alpha_2}(\R^n)}+ [w_2]_{\beta;U}^{(-\alpha_2)}\right),
\end{equation}
for all functions $w_1, w_2$ with finite right hand side. The constant $C$ depends only on $\alpha_1$, $\alpha_2$, $n$, $\beta$, and $s$.
\end{lem}

\begin{proof}
Let $x_0\in U$ and $R=d_{x_0}/2$, and denote $B_{\rho}=B_\rho(x_0)$. Let
\[K=\left( [w_1]_{C^{\alpha_1}(\R^n)}+ [w_1]_{\beta;U}^{(-\alpha_1)}\right)\left( [w_2]_{C^{\alpha_2}(\R^n)}+ [w_2]_{\beta;U}^{(-\alpha_2)}\right)\,.\]

First we bound $|I_s(w_1,w_2)(x_0)|$.
\[
\begin{split}
|I_s(w_1,w_2)(x_0)|&\le C\int_{\R^n}\frac{\bigl|w_1(x_0)-w_1(y)\bigr|\bigl|w_2(x_0)-w_2(y)\bigr|}{|x_0-y|^{n+2s}}\,dy \\
&\le C\int_{B_R(0)} \frac{R^{\alpha_1+\alpha_2-2\beta}[w_1]_{\beta;U}^{(-\alpha_1)}[w_2]_{\beta;U}^{(-\alpha_2)}|z|^{2\beta}}{|z|^{n+2s}}\,dz \ +\\
&\qquad\qquad\qquad\quad+  C\int_{\R^n\setminus B_R(0)} \frac{[w_1]_{C^{\alpha_1}(\R^n)}[w_2]_{C^{\alpha_2}(\R^n)} |z|^{\alpha_1+\alpha_2} }{|z|^{n+2s}}\,dz \\
&\le C R^{\alpha_1+\alpha_2-2s}K \,.
\end{split}
\]

Let $x_1,x_2\in B_{R/2}(x_0)\subset B_{2R}(x_0)$.
Next, we bound $|I_s(w_1,w_2)(x_1)-I_s(w_1,w_2)(x_2)|$.
Let $\eta$ be a smooth cutoff function such that $\eta\equiv 1$ on $B_{1}(0)$ and $\eta\equiv 0$ outside $B_{3/2}(0)$. Define
\[\eta^R(x)=\eta\left(\frac{x-x_0}{R}\right)\quad \mbox{ and  } \quad \bar w_i= \bigr(w_i-w_i(x_0)\bigl)\eta^R\,,\quad i=1,2\,.\]
Note that we have
\[
\|\bar w_i\|_{L^\infty(\R^n)}= \|\bar w_i\|_{L^\infty({B_{3R/2}})} \le \left(\frac{3R}{2}\right)^{\alpha_i} [w_i]_{C^{\alpha_i}(\R^n)}
\]
and
\[
\begin{split}
[\bar w_i]_{C^{\beta}(\R^n)} &\le  C \biggl([w_i]_{C^{\beta}(\overline{B_{3R/2}})}\|\eta\|_{L^\infty(B_{3R/2})}  + \|w_i-w_i(0)\|_{L^\infty(B_{3R/2})} [w_i]_{C^{\beta}(\overline{B_{3R/2}})}\biggr)
\\
&\le C R^{\alpha_i-\beta} \biggl( [w_i]_{C^{\alpha_i}(\R^n)}+ [w_i]_{\beta;U}^{(-\alpha_i)}\biggr)\,.
\end{split}
\]

Let
\[\varphi_i = w_i- w_i(x_0)- \bar w_i\] and observe that $\varphi_i$ vanishes in $B_R$. Hence, $\varphi_i(x_1)= \varphi_i(x_2)=0$,  $i=1,2$.
Next, let us write
\[ I_s(w_1,w_2)(x_1)-I_s(w_1,w_2)(x_2) = c_{n,s}\left( J_{11} + J_{12}+J_{21} +J_{22}\right),\]
where
\[\begin{split}
J_{11} = \int_{\R^n} \frac{\bigl(\bar w_1(x_1)-\bar w_1(y)\bigr)\bigl(\bar w_2(x_1)-\bar w_2(y)\bigr)}{|x_1-y|^{n+2s}}&\,dy\\
&\hspace{-20mm}-\int_{\R^n}\frac{\bigl(\bar w_1(x_2)-\bar w_1(y)\bigr)\bigl(\bar w_2(x_2)-\bar w_2(y)\bigr)}{|x_2-y|^{n+2s}}\,dy\,,
\end{split}\]
\[J_{12} = \int_{\R^n\setminus B_{R}} \frac{-\bigl(\bar w_1(x_1)-\bar w_1(y)\bigr)\varphi_2(y)}{|x_1-y|^{n+2s}}+
\frac{\bigl(\bar w_1(x_2)-\bar w_1(y)\bigr)\varphi_2(y)}{|x_2-y|^{n+2s}}\,dy\,,\]
\[J_{21} = \int_{\R^n\setminus B_{R}} \frac{-\bigl(\bar w_2(x_1)-\bar w_2(y)\bigr)\varphi_1(y)}{|x_1-y|^{n+2s}}+
\frac{\bigl(\bar w_2(x_2)-\bar w_2(y)\bigr)\varphi_1(y)}{|x_2-y|^{n+2s}}\,dy\,,\]
and
\[J_{22} = \int_{\R^n\setminus B_{R}} \frac{\varphi_1(y) \varphi_2(y)}{|x_1-y|^{n+2s}}-
\frac{\varphi_1(y)\varphi_2(y)}{|x_2-y|^{n+2s}}\,dy\,.\]
We now bound separately each of these terms.

{\em Bound of $J_{11}$.} We write $J_{11} = J_{11}^1+J_{11}^2$ where
\[J_{11}^1= \int_{\R^n}\frac{\bigl(\bar w_1(x_1)-\bar w_1(x_1+z)-\bar w_1(x_2)+\bar w_1(x_2+z)\bigr)
\bigl(\bar w_2(x_1)-\bar w_2(x_1+z)\bigr)}{|z|^{n+2s}}\,dz,  \]
\[J_{11}^2= \int_{\R^n}\frac{\bigl(\bar w_1(x_2)-\bar w_1(x_2+z)\bigr)\bigl(\bar w_2(x_1)-\bar w_2(x_1+z)-\bar w_2(x_2)+\bar w_2(x_2+z)\bigr)
}{|z|^{n+2s}}\,dz\,.\]

To bound $|J_{11}^1|$ we proceed as follows
\[
\begin{split}
|J_{11}^1| &\le \int_{B_r(0)} \frac{R^{\alpha_1-\beta}[w_1]_{\beta;U}^{(-\alpha_1)}|z|^{\beta}R^{\alpha_2-\beta}[w_2]_{\beta;U}^{(-\alpha_2)}|z|^{\beta} }{|z|^{n+2s}}\,dz\,+\\
&\hspace{20mm}+ \int_{\R^n\setminus B_r(0)} \frac{R^{\alpha_1-\beta}[w_1]_{\beta;U}^{(-\alpha_1)} r^{\beta} R^{\alpha_2-\beta}[w_2]_{\beta;U}^{(-\alpha_2)}|z|^{\beta} }{|z|^{n+2s}}\,dz  \\
&\le  C R^{\alpha_1+\alpha_2-2\beta}r^{2\beta-2s} K\,.
\end{split}
\]
Similarly, $|J_{11}^2|\le   C R^{\alpha_1+\alpha_2-2\beta}r^{2\beta-2s} K$.

{\em Bound of $J_{12}$ and $J_{21}$.}
We write $J_{12}= J_{12}^1+ J_{12}^2$ where
\[ J_{12}^1 = \int_{\R^n \setminus B_R} -\varphi_2(y) \frac{\bar w_1(x_1)-\bar w_1(x_2)}{|x_1-y|^{n+2s}}\,dy
\]
and
\[ J_{12}^2 =
\int_{\R^n \setminus B_R} -\varphi_2(y) \bigl(\bar w_1(x_2)-\bar w_1(y)\bigr)\left\{\frac{1}{|x_1-y|^{n+2s}}-\frac{1}{|x_2-y|^{n+2s}}\right\}\,dy\,.
\]
To bound $|J_{12}^1|$ we recall that $\varphi_2(x_1)=0$ and proceed as follows
\[
\begin{split}
|J_{12}^1|&\le C\int_{\R^n \setminus B_R} |x_1-y|^{\alpha_2}[\varphi_2]_{C^{0,\alpha_2}(\R^n)} \frac{R^{\alpha_1-\beta} [w_1]_{\beta;U}^{(-\alpha_1)} r^\beta}{|x_1-y|^{n+2s}}\,dy\\
&\leq CR^{\alpha_1+\alpha_2-\beta-2s}r^{\beta} K \le CR^{\alpha_1+\alpha_2-2\beta}r^{2\beta-2s} K.
\end{split}
\]
We have used that $[\varphi_2]_{C^{\alpha_2}(\R^n)}=[w-\bar w]_{C^{\alpha_2}(\R^n)}\leq 2[w]_{C^{\alpha_2}(\R^n)}$, $r\le R$, and $\beta<2s$.

To bound $|J_{12}^2|$, let $\Phi(z)=|z|^{-n-2s}$. Note that, for each $\gamma\in(0,1]$, we have
\begin{equation}\label{boundPhi}|\Phi(z_1-z)-\Phi(z_2-z)| \le C |z_1-z_2|^{\gamma} |z|^{-n-2s-\gamma}\end{equation}
for all $z_1, z_2$ in $B_{R/2}(0)$ and $z\in\R^n\setminus B_R(0)$.
Then, using that $\varphi_2(x_2)=0$,
\[
\begin{split}
|J_{12}^2|&\le C\int_{\R^n \setminus B_R} |x_2-y|^{\alpha_1+\alpha_2}[\varphi_2]_{C^{\alpha_2}(\R^n)}[\varphi_2]_{C^{\alpha_2}(\R^n)}
\frac{|x_1-x_2|^{2\beta-2s}}{|x_2-y|^{n+2\beta}}\,dy\\
&\leq C R^{\alpha_1+\alpha_2-2\beta}r^{2\beta-2s}K\,.
\end{split}
\]

This proves that $|J_{12}|\le C R^{\alpha_1+\alpha_2-2\beta}r^{2\beta-2s}K$. Changing the roles of $\alpha_1$ and $\alpha_2$ we obtain the same bound for $|J_{21}|$.

{\em Bound of $J_{22}$.} Using again $\varphi_i(x_i)=0$, $i=1,2$, we write
\[J_{22} = \int_{\R^n\setminus B_{R}} \bigl(\varphi_1(x_1) -\varphi_1(y)\bigr)\bigl(\varphi_2(x_1) -\varphi_2(y)\bigr)\left(\frac{1}{|x_1-y|^{n+2s}}-\frac{1}
{|x_2-y|^{n+2s}}\right)dy\,.\]
Hence, using again \eqref{boundPhi},
\[
\begin{split}
|J_{22}|&\le C\int_{\R^n \setminus B_R} |x_1-y|^{\alpha_1+\alpha_2}[\varphi_2]_{C^{0,\alpha_2}(\R^n)}[\varphi_2]_{C^{0,\alpha_2}(\R^n)}
\frac{|x_1-x_2|^{2\beta-2s}}{|x_1-y|^{n+2\beta}}\,dy\\
&\leq C R^{\alpha_1+\alpha_2-2\beta}r^{2\beta-2s}K\,.
\end{split}
\]

Summarizing, we have proven that for all $x_0$ such that $d_x=2R$ and for all $x_1,x_2\in B_{R/2}(x_0)$ it holds
\[
|I_s(\delta_0^s,w)(x_0)|\le  C R^{\alpha_1-\alpha_2-2s}K
\]
and
\[\frac{|I_s(\delta_0^s,w)(x_1)-I_s(\delta_0^s,w)(x_2)|}{|x_1-x_2|^{2\beta-2s}}\le C R^{\alpha_1+\alpha_2-2\beta}\bigl( [w]_{\alpha+s;U}^{(-\alpha)}+[w]_{C^\alpha(\R^n)}\bigr)\,.\]
This yields \eqref{eq:bound-Iclaim}, as shown in Step 2 in the proof of Lemma \ref{refined-2s-gain}.
\end{proof}

Next we prove Lemma \ref{lem:bound-I}.

\begin{proof}[Proof of Lemma \ref{lem:bound-I}]
The distance function $\delta_0$ is $C^{1,1}$ in $\overline{\Omega_{\rho_0}}$ and since $U\subset \Omega_{\rho_0}$ we have $d_x\le \delta_0(x)$ for all $x\in U$.
Hence, it follows that
\[[\delta_0^s]_{C^{s}(\R^n)}+ [\delta_0^s]_{\beta;U}^{(-s)} \le C(\Omega,\beta)\]
for all $\beta\in[s,2]$.

Then, applying Lemma \ref{claim:I} with $w_1=w$, $w_2=\delta_0^s$, $\alpha_1=\alpha$, $\alpha_2=s$, and $\beta=s+\alpha$, we obtain
\[\|I_s(w,\delta_0^s)\|_{2\alpha;U}^{(s-\alpha)}\le C\biggl( [w]_{C^{\alpha}(\R^n)}+ [w]_{\alpha+s;U}^{(-\alpha)}\biggr)\,,\]
and hence \eqref{eq:bound-I} follows.
\end{proof}

Using Lemma \ref{lem:bound-I} we can now prove Theorem \ref{thm:int-est-v} and Corollary \ref{krylov}.

\begin{proof}[Proof of Theorem \ref{thm:int-est-v}]
Let $U\subset\subset \Omega_{\rho_0}$.
We prove first that there exist $\alpha\in(0,1)$ and $C$, depending only on $s$ and $\Omega$ ---and not on $U$---, such that
\[\|u/\delta^s\|_{\alpha+2s;U}^{(-\alpha)}\le C\left(\|g\|_{L^\infty(\Omega)}+\|g\|_{\alpha;\Omega}^{(s-\alpha)}\right).\]
Then, letting $U\uparrow \Omega_{\rho_0}$ we will find that this estimate holds in $\Omega_{\rho_0}$ with the same constant.

To prove this, note that by Theorem \ref{thm:v-is-Calpha} we have
\[
\|u/\delta^s\|_{C^\alpha(\overline\Omega)} \le
C\bigl(s,\Omega\bigr)\|g\|_{L^\infty(\Omega)}\,.
\]
Recall that $v$ denotes the $C^\alpha(\R^n)$ extension of $u/\delta^s|_\Omega$ given by Lemma \ref{prop:extension-op-E}, which satisfies $\|v\|_{C^\alpha(\R^n)}=\|u/\delta^s\|_{C^\alpha(\overline\Omega)}$.
Since $u\in C^{\alpha+2s}(\Omega)$ and $\delta\in C^{1,1}(\Omega_{\rho_0})$, it is clear that $\|v\|_{\alpha+2s;U}^{(-\alpha)} <\infty$
---it is here where we use that we are in a subdomain $U$ and not in $\Omega_{\rho_0}$.
Next we obtain an  a priori bound for this seminorm in $U$. To do it, we use the equation \eqref{equaciov} for $v$:
\[(-\Delta)^sv= \frac{1}{\delta^s}\biggl( g(x) - v (-\Delta)^s \delta_0^s + I(\delta_0^s,v)\biggr) \quad\mbox{in } \Omega_{\rho_0}=\{x\in \Omega\,:\,\delta(x)<\rho_0\}\,.\]

Now we will se that this equation and Lemma \ref{refined-2s-gain} lead to an a priori bound for $\|v\|_{\alpha+2s;U}^{(-\alpha)}$. To apply Lemma \ref{refined-2s-gain}, we need to bound $\|(-\Delta)^s v\|_{\alpha;U}^{(2s-\alpha)}$. Let us examine the three terms on the right hand side of the equation.

{\em First term.} Using that
\[d_x= {\rm dist}(x,\partial U)<  {\rm dist}(x,\partial \Omega)=\delta(x)\]
for all $x\in U$ we obtain that, for all $\alpha\le s$,
\[\|\delta^{-s}g\|_{\alpha;U}^{(2s-\alpha)}\le C\bigl(s,\Omega\bigr)\|g\|_{\alpha;\Omega}^{(s-\alpha)}\,.\]

{\em Second term.}  We know from Lemma \ref{lapsdeltas} that, for $\alpha\le \min\{s,1-s\}$,
\[\|(-\Delta)^s \delta_0^s\|_{C^{\alpha}(\overline{\Omega_{\rho_0}})} \le C\bigl( s, \Omega)\,.\]
Hence,
\[
\begin{split}
\|\delta^{-s}v(-\Delta)^s \delta_0^s\|_{\alpha;U}^{(2s-\alpha)} &\le {\rm diam} (\Omega)^s \|\delta^{-s}v(-\Delta)^s \delta_0^s\|_{\alpha;U}^{(s-\alpha)} \le C\bigl(s, \Omega\bigr) \|v\|_{C^\alpha(\R^n)}\\
&\le C\bigl(s,\Omega\bigr)\|g\|_{L^\infty(\Omega)}\,.
\end{split}
\]

{\em Third term.}  From Lemma \ref{lem:bound-I} we know that
\[
\|I(v,\delta_0^s)\|_{\alpha;U}^{(s-\alpha)}\le C(n,s,\alpha)\biggl( \|v\|_{C^\alpha(\R^n)}+ [v]_{\alpha+s;U}^{(-\alpha)}\biggr)\,,
\]
and hence
\[
\begin{split}
\|\delta^{-s}I(v,\delta_0^s)\|_{\alpha;U}^{(2s-\alpha)}&\le C(n,s,\Omega,\alpha)\biggl( \|v\|_{C^\alpha(\R^n)}+ [v]_{\alpha+s;U}^{(-\alpha)}\biggr) \\
&\le C(n,s,\Omega,\alpha,\varepsilon_0)\|v\|_{C^\alpha(\R^n)}+ \varepsilon_0 \|v\|_{\alpha+2s;U}^{(-\alpha)}
\end{split}
\]
for each $\epsilon_0>0$.
The last inequality is by standard interpolation.

Now,  using Lemma \ref{refined-2s-gain} we deduce
\[\begin{split}
\|v\|_{\alpha+2s;U}^{(-\alpha)}&\le C\left(\|v\|_{C^\alpha(\R^n)} + \|(-\Delta)^s v\|_{\alpha;U}^{(2s-\alpha)}\right)\\
&\hspace{-8mm}\leq C\left(\|v\|_{C^\alpha(\R^n)}+\|\delta^{-s}g\|_{\alpha;U}^{(2s-\alpha)}+\|\delta^{-s}v(-\Delta)^s \delta_0^s\|_{\alpha;U}^{(2s-\alpha)}+\|I(v,\delta_0^s)\|_{\alpha;U}^{(s-\alpha)}\right)\\
&\hspace{-8mm}\le C(s,\Omega,\alpha,\varepsilon_0)\left(\|g\|_{L^\infty(\Omega)}+\|g\|_{\alpha;\Omega}^{(s-\alpha)}\right)+ C\varepsilon_0 \|v\|_{\alpha+2s;U}^{(-\alpha)},
\end{split}\]
and choosing $\varepsilon_0$ small enough we obtain
\[\|v\|_{\alpha+2s;U}^{(-\alpha)}\le C\left(\|g\|_{L^\infty(\Omega)}+\|g\|_{\alpha;\Omega}^{(s-\alpha)}\right).\]
Furthermore, letting $U\uparrow\Omega_{\rho_0}$ we obtain that the same estimate holds with $U$ replaced by $\Omega_{\rho_0}$.

Finally, in $\Omega\setminus\Omega_{\rho_0}$ we have that $u$ is $C^{\alpha+2s}$ and $\delta^s$ is uniformly positive and $C^{0,1}$. Thus, we have $u/\delta^s\in C^\gamma(\Omega\setminus\Omega_{\rho_0})$, where $\gamma=\min\{1,\alpha+2s\}$, and the theorem follows.
\end{proof}

Next we give the

\begin{proof}[Proof of Corollary \ref{krylov}]
(a) It follows from Proposition \ref{prop:u-is-Cs} that $u\in C^s(\R^n)$. The interior estimate follow by applying repeatedly Proposition \ref{prop:int-est-u}.

(b) It follows from Theorem \ref{thm:v-is-Calpha} that $u/\delta^s|_\Omega\in C^\alpha(\overline\Omega)$. The interior estimate follows from Theorem \ref{thm:int-est-v}.
\end{proof}

The following two lemmas are closely related to Lemma \ref{claim:I} and are needed in \cite{RS} and in Remark \ref{remviscosity} of this paper.

\begin{lem}\label{cosaiscalpha}
Let $U$ be an open domain and $\alpha$ and $\beta$ be such that $\alpha\le s<\beta$ and $\beta-s$ is not an integer. Let $k$ be an integer such that $\beta= k+\beta'$ with $\beta'\in(0,1]$. Then,
\begin{equation}\label{eq:bound-lap-s/2}
[(-\Delta)^{s/2}w]_{\beta-s;U}^{(s-\alpha)}\le C\bigl( \|w\|_{C^\alpha(\R^n)}+ \|w\|_{\beta;U}^{(-\alpha)}\bigr)\,,
\end{equation}
for all $w$ with finite right hand side.
The constant $C$ depends only on $n$, $s$, $\alpha$, and $\beta$.
\end{lem}

\begin{proof}
Let $x_0\in U$ and $R=d_{x_0}/2$, and denote $B_{\rho}=B_\rho(x_0)$.
Let $\eta$ be a smooth cutoff function such that $\eta\equiv 1$ on $B_{1}(0)$ and $\eta\equiv 0$ outside $B_{3/2}(0)$. Define
\[\eta^R(x)=\eta\left(\frac{x-x_0}{R}\right)\quad \mbox{ and  } \quad \bar w= \bigr(w-w(x_0)\bigl)\eta^R\,.\]
Note that we have
\[
\|\bar w\|_{L^\infty(\R^n)}= \|\bar w\|_{L^\infty(\overline{B_{3R/2}})} \le \left(\frac{3R}{2}\right)^\alpha [w]_{C^\alpha(\R^n)}\,.
\]
In addition, for each $1\le l\le k$
\[
\begin{split}
\|D^l \bar w\|_{L^\infty(\R^n)} &\le C\sum_{m=0}^l \|D^m (w-w(x_0)) D^{l-m} \eta^R\|_{L^\infty(\overline{B_{3R/2}})} \\
&\le C R^{-l+\alpha} \left( [w]_{C^\alpha(\R^n)}+\sum_{m=1}^{l} [w]_{m,U}^{(-\alpha)} \right).
\end{split}
\]
Hence, by interpolation, for each $0\le l\le k-1$
\[
\|D^l \bar w\|_{C^{l+\beta'}(\R^n)}\le C R^{-l-\beta'+\alpha} \left( [w]_{C^\alpha(\R^n)}+\sum_{m=1}^{l} [w]_{m,U}^{(-\alpha)} \right)\,,
\]
and therefore
\begin{equation}\label{Dk-bar-w}
[D^k \bar w]_{C^{\beta'}(\R^n)} \le C R^{-\beta+\alpha}\|w\|_{\beta;U}^{(-\alpha)}\,.
\end{equation}

Let $\varphi = w- w(x_0)- \bar w$ and observe that $\varphi$ vanishes in $B_R$ and, hence, $\varphi(x_1)= \varphi(x_2)=0$.

Next we proceed differently if $\beta'>s$ or if $\beta'< s$. This is because $C^{\beta-s}$ equals either $C^{k,\beta'-s}$ or $C^{k-1,1+\beta'-s}$.

{\em Case 1.} Assume $\beta'> s$. Let $x_1,x_2\in B_{R/2}(x_0)\subset B_{2R}(x_0)$. We want to bound $|D^k (-\Delta)^{s/2} w(x_1)- D^k (-\Delta)^{s/2} w(x_2)|$, where $D^k$ denotes any
$k$-th derivative with respect to a fixed multiindex.
We have
\[ (-\Delta)^{s/2} w = (-\Delta)^{s/2} \bar w + (-\Delta)^{s/2} \varphi \quad \mbox{in }B_{R/2}\,.\]
Then,
\[D^k(-\Delta)^{s/2}w(x_1) - D^k(-\Delta)^{s/2}w(x_2) = c_{n,\frac s2}(J_1 + J_2)\,,\]
where
\[J_1 = \int_{\R^n} \left\{\frac{ D^k \bar w(x_1)- D^k \bar w(y)}{|x_1-y|^{n+s}}-
\frac{D^k \bar w(x_2)- D^k \bar w(y)}{|x_2-y|^{n+s}}\right\}dy\,\]
and
\[J_2 = D^k\int_{\R^n\setminus B_{R}} \frac{-\varphi(y)}{|x_1-y|^{n+s}}\,dy- D^k\int_{\R^n\setminus B_{R}}
\frac{ -\varphi(y)}{|x_2-y|^{n+s}}\,dy\,.\]

To bound $|J_{1}|$ we proceed as follows. Let $r=|x_1-x_2|$. Then, using \eqref{Dk-bar-w},
\[
\begin{split}
|J_{1}|&= \biggl|\int_{\R^n}\frac{D^k\bar w(x_1)- D^k\bar w(x_1+z)-D^k\bar w(x_2)+D^k\bar w(x_2+z)}{|z|^{n+s}}\,dz \biggr| \\
&\le \int_{B_r} \frac{R^{\alpha-\beta}\|w\|_{\beta;U}^{(-\alpha)}  |z|^{\beta'}  }{  |z|^{n+s}  }\,dz
+ \int_{\R^n\setminus B_r} \frac{R^{\alpha-\beta}\|w\|_{\beta;U}^{(-\alpha)} r^{\beta'}}{|z|^{n+s}}\,dz  \\
&\le  CR^{\alpha-\beta}r^{\beta'-s}\|w\|_{\beta;U}^{(-\alpha)}\,.
\end{split}
\]

Let us bound now $|J_2|$.
Writing $\Phi(z)= |z|^{-n-s}$ and using that $\varphi(x_0)=0$,
\[
\begin{split}
|J_{2}|&=
\biggl|\int_{ \R^n\setminus B_{R} } \varphi(y) \bigl( D^k \Phi (x_1-y)- D^k\Phi(x_2-y) \bigr) \,dy \biggr| \\
&\le C\int_{\R^n \setminus B_R} |x_0-y|^\alpha [w]_{C^{\alpha}(\R^n)}
\frac{|x_1-x_2|^{\beta'-s}}{|x_0-y|^{n+\beta}}\,dy \\
&\le C R^{\alpha-\beta} r^{\beta'-s} [w]_{C^{\alpha}(\R^n)},
\end{split}
\]
where we have used that
\[|D^k\Phi(z_1-z)-D^k \Phi(z_2-z)| \le C |z_1-z_2|^{\beta'-s} |z|^{-n-\beta}\]
for all $z_1, z_2$ in $B_{R/2}(0)$ and $z\in\R^n\setminus B_R$.

Hence, we have proved that
\[[(-\Delta)^{s/2}w]_{C^{\beta-s}(\overline{B_R(x_0)})}\leq CR^{\alpha-\beta}\|w\|_{\beta;U}^{(-\alpha)}.\]

{\em Case 2.}  Assume $\beta'< s$. Let $x_1,x_2\in B_{R/2}(x_0)\subset B_{2R}(x_0)$. We want to bound $|D^{k-1} (-\Delta)^{s/2} w(x_1)- D^{k-1} (-\Delta)^{s/2} w(x_2)|$. We proceed as above but we now use
\[
\begin{split}
|D^{k-1}\bar w(x_1)&- D^{k-1}\bar w(x_1+y)-D^{k-1}\bar w(x_2)+D^{k-1}\bar w(x_2+y)|\le \\
&\le
\left|D^{k}\bar w(x_1)-D^{k}\bar w(x_2)\right||y| + |y|^{1+\beta'}
\|\bar w\|_{C^{\beta}(\R^n)} \\
&\le \bigl(|x_1-x_2|^{\beta'} |y|+ |y|^{1+\beta'}\bigr) R^{\alpha-\beta}\|w\|_{\beta;U}^{(-\alpha)}
\end{split}
\]
in $B_r$, and
\[
\begin{split}
|D^{k-1}\bar w(x_1)&- D^{k-1}\bar w(x_1+y)-D^{k-1}\bar w(x_2)+D^{k-1}\bar w(x_2+y)|\le \\
&\le
\left|D^{k}\bar w(x_1)-D^{k}\bar w(x_1+y)\right| |x_1-x_2| + |x_1-x_2|^{1+\beta'}
\|\bar w\|_{C^{\beta}(\R^n)} \\
&\le \bigl(|y|^{\beta'} |x_1-x_2|+ |x_1-x_2|^{1+\beta'}\bigr) R^{\alpha-\beta}\|w\|_{\beta;U}^{(-\alpha)}
\end{split}
\]
in $\R^n\backslash B_r$.
Then, as in Case 1 we obtain $[(-\Delta)^{s/2}w]_{C^{\beta-s}(\overline{B_R(x_0)})}\leq CR^{\alpha-\beta}\|w\|_{\beta;U}^{(-\alpha)}$.

This yields \eqref{eq:bound-lap-s/2}, as in Step 2 of Lemma \ref{refined-2s-gain}.
\end{proof}

Next lemma is a variation of the previous one and gives a pointwise bound for $(-\Delta)^{s/2} w$.
It is used in Remark \ref{remviscosity}.

\begin{lem}\label{remlog} Let $U\subset\R^n$ be an open set, and let $\beta>s$. 
Then, for all $x\in U$
\[|(-\Delta)^{s/2}w(x)|\leq C(\|w\|_{C^s(\R^n)}+\|w\|_{\beta;U}^{(-s)})\biggl(1+|\log{\rm dist}(x,\partial U)|\biggr),\]
whenever $w$ has finite right hand side.
The constant $C$ depends only on $n$, $s$, and $\beta$.
\end{lem}

\begin{proof} We may assume $\beta<1$. Let $x_0\in U$ and $R=d_{x_0}/2$, and define $\bar w$ and $\varphi$ as in the proof of the previous lemma. Then,
\[ (-\Delta)^{s/2} w(x_0) =(-\Delta)^{s/2} \bar w(x_0) +(-\Delta)^{s/2} \varphi(x_0) = c_{n,\frac s2} (J_1+J_2),\]
where
\[
J_1= \int_{\R^n} \frac{\bar w(x_0)-\bar w(x_0+z)}{|z|^{n+s}}\,dz \quad \mbox{and}\quad
J_2= \int_{\R^n\setminus B_R} \frac{-\varphi(x_0+z)}{|z|^{n+s}}\,dz.
\]
With similar arguments as in the previous proof we readily obtain $|J_1|\le C (1+|\log R|) \|w\|_{\beta;U}^{(-s)}$ and $|J_2|\le C (1+|\log R|) \|w\|_{C^s(\R^n)}$.
\end{proof}

\appendix
\section{Basic tools and barriers}

In this appendix we prove Proposition \ref{prop:solution} and Lemmas \ref{prop:subsolution} and \ref{prop:supersolution}. Proposition \ref{prop:solution} is well-known (see \cite{CRS}), but for the sake of completeness we sketch here a proof that uses the Caffarelli-Silvestre extension problem \cite{CSext}.

\begin{proof}[Proof of Proposition \ref{prop:solution}]
Let $(x,y)$ and $(r,\theta)$ be Cartesian and polar coordinates of the plane. 
The coordinate $\theta\in(-\pi,\pi)$ is taken so that $\{\theta=0\}$ on $\{y=0,\  x>0\}$. 
Use that the function $r^s \cos(\theta/2)^{2s}$ is a solution in the half-plane $\{y>0\}$ to the extension problem \cite{CSext},
\[{\rm div}(y^{1-2s}\nabla u) = 0 \quad \mbox{ in} \ \{y>0\},\]
and that its trace on $y=0$ is $\varphi_0$.
\end{proof}

The fractional Kelvin transform has been studied thoroughly in \cite{Bog}.

\begin{prop}[Fractional Kelvin transform]\label{prop:frac-kelvin}
 Let $u$ be a smooth bounded function in $\R^n\setminus\{0\}$. Let $x\mapsto x^*= x/|x|^2$ be the inversion with respect to the unit sphere. Define $u^*(x)= |x|^{2s-n} u(x^*)$. Then,
\begin{equation}\label{eq:frac-kelvin}
(-\Delta)^s u^*(x) = |x|^{-2s-n} (-\Delta)^s u(x^*)\,,
\end{equation}
for all $x\neq 0$.
\end{prop}

\begin{proof}
Let $x_0\in\R^n\setminus\{0\}$. By subtracting a constant to $u^*$ and using $(-\Delta)^s |x|^{2s-n}=0$ for $x\neq 0$, we may assume $u^*(x_0)= u(x_0^*)=0$. Recall that
\[|x-y| = \frac{|x^*-y^*|}{|x^*||y^*|}\,.\]
Thus, using the change of variables $z=y^*= y/|y|^2$,
\[
\begin{split}
(-\Delta)^s u^*(x_0) &= c_{n,s}\ \text{PV} \int_{\R^n} \frac{- u^*(y)}{|x_0-y|^{n+2s}}\,dy\\
&= c_{n,s}\ \text{PV} \int_{\R^n} \frac{- |y|^{2s-n} u(y^*)}{|x_0^*-y^*|^{n+2s}} |x_0^*|^{n+2s}|y^*|^{n+2s}\,dy\\
&= c_{n,s} |x_0|^{-n-2s}\ \text{PV} \int_{\R^n} \frac{- |z|^{n-2s} u(z)}{|x_0^*-z|^{n+2s}} |z|^{n+2s}\,|z|^{-2n} dz\\
&= c_{n,s} |x_0|^{-n-2s}\ \text{PV} \int_{\R^n} \frac{-  u(z)}{|x_0^*-z|^{n+2s}} dz\\
&= |x_0|^{-n-2s} (-\Delta)^s u(x_0^*)\,.
\end{split}
\]
\end{proof}

Now, using Proposition \ref{prop:frac-kelvin} we prove Lemma \ref{prop:supersolution}.

\begin{proof}[Proof of Lemma \ref{prop:supersolution}]
Let us denote by $\psi$ (instead of $u$) the explicit solution \eqref{explicitsolution2} to problem \eqref{explicitsolution1} in $B_1$, which satisfies
\begin{equation}\label{supersol-in-ball}
\begin{cases}
(-\Delta)^s\psi = 1\quad&\mbox{in }B_1\\
\psi\equiv 0 &\mbox{in }\R^n \setminus B_1\\
0 <\psi< C(1-|x|)^s &\mbox{in }B_1\,.\\
\end{cases}
\end{equation}

From $\psi$, the supersolution $\varphi_1$ in the exterior of the ball is readily built using the fractional Kelvin transform. Indeed, let $\xi$ be a radial smooth function satisfying $\xi \equiv 1$ in $\R^n\setminus B_5$ and $\xi\equiv 0$ in $B_4$, and define $\varphi_1$ by
\begin{equation}\label{varphi1}
\varphi_1 (x) = C |x|^{2s-n}\psi(1-|x|^{-1}) + \xi(x)\,.
\end{equation}
Observe that $(-\Delta)^s \xi \ge - C_2$ in $B_4$, for some $C_2>0$.  Hence, if we take $C\ge 4^{2s+n}(1+C_2)$, using \eqref{eq:frac-kelvin}, we have
\[(-\Delta)^s\varphi_1 (x) \ge C |x|^{-2s-n} + (-\Delta)^s \xi(x) \ge 1 \quad \mbox{in } B_4 \,.\]
Now it is immediate to verify that $\varphi_1$ satisfies  \eqref{eq:propsupersol} for some $c_1>0$.

To see that $\varphi_1\in H^s_{\rm loc}(\R^n)$ we observe that from \eqref{varphi1} it follows
\[ |\nabla \varphi_1(x)|\le C(|x|-1)^{s-1} \quad \mbox{in }\R^n \setminus B_1\]
and hence, using Lemma \ref{remlog}, we have $(-\Delta)^{s/2}\varphi_1\in L^p_{\rm loc}(\R^n)$ for all $p<\infty$.
\end{proof}

Next we prove Lemma \ref{prop:subsolution}.

\begin{proof}[Proof of Lemma \ref{prop:subsolution}]
We define
\[\psi_1(x)= (1-|x|^2)^s\chi_{B_1}(x)\,.\] 
Since \eqref{explicitsolution2} is the solution of problem \eqref{explicitsolution1}, we have $(-\Delta)^s\psi_1$ is bounded in $B_1$.
Hence, for $C>0$ large enough the function $\psi= \psi_1 + C\chi_{\overline{B_{1/4}}}$ satisfies $(-\Delta)^s \psi\le 0$ in $B_1\setminus \overline{B_{1/4}}$ and it can be used as a viscosity subsolution. 
Note that $\psi$ is upper semicontinuous, as required to viscosity subsolutions, and it satisfies pointwise (if $C$ is large enough)
\[
\begin{cases}
 \psi  \equiv 0 \quad &\mbox{in }\R^n\setminus B_1 \\
(-\Delta)^s \psi  \le 0  &\mbox{in }B_1\setminus \overline{B_{1/4}}\\
\psi  = 1 &\mbox{in }\overline{B_{1/4}}\\
\psi (x)\ge c(1-|x|)^s & \mbox{in } B_1.
\end{cases}
\]

If we want a subsolution which is continuous and $H^s(\R^n)$ we may construct it as follows. We consider the viscosity solution (which is also a weak solution by Remark \ref{remviscosity}) of
\[
\begin{cases}
(-\Delta)^s\varphi_2  = 0  &\mbox{in }B_1\setminus B_{1/4}\\
\varphi_2  \equiv 0 \quad &\mbox{in }\R^n\setminus B_1 \\
\varphi_2  = 1 &\mbox{in }\overline{B_{1/4}}.
\end{cases}
\]
Using $\psi$ as a lower barrier, it is now easy to prove that $\varphi_2$ satisfies \eqref{eq:propsubsol} for some constant $c_2>0$.
\end{proof}

\section*{Acknowledgements}

The authors thank Xavier Cabr\'e for his guidance and useful discussions on the topic of this paper.


\begin{thebibliography}{00}

\bibitem{BMW} M. Birkner, J. A. L\'opez-Mimbela, A. Wakolbinger, \emph{Comparison results and steady states for the Fujita equation with fractional Laplacian}, Ann. Inst. H. Poincar\'e Anal. Non Lin\'eaire 22 (2005), 83–97.

\bibitem{B} K. Bogdan, \emph{The boundary Harnack principle for the fractional Laplacian}, Studia Math. 123 (1997), 43–80.

\bibitem{BGR} K. Bogdan, T. Grzywny, M. Ryznar, \emph{Heat kernel estimates for the fractional Laplacian with Dirichlet conditions}, Ann. of Prob. 38 (2010), 1901-1923.

\bibitem{BKK} K. Bogdan, T. Kulczycki, M. Kwa\'snicki, \emph{Estimates and structure of $\alpha$-harmonic functions}, Probab. Theory Related Fields 140 (2008), 345–381.

\bibitem{Bog} K. Bogdan, T. Zak, \emph{On Kelvin transformation}, J. Theoret. Probab. 19 (2006), 89-120.

\bibitem{CaffC} L. Caffarelli, X. Cabré, \emph{Fully Nonlinear Elliptic Equations}, American Mathematical Society Colloquium Publications 43, American Mathematical Society, Providence, RI, 1995.

\bibitem{CRS} L. Caffarelli, J. M. Roquejoffre, Y. Sire, \emph{Variational problems in free boundaries for the fractional Laplacian}, J. Eur. Math. Soc. 12 (2010), 1151-1179.

\bibitem{CSext} L. Caffarelli, L. Silvestre, \emph{An extension problem related to the fractional Laplacian}, Comm. Partial Differential Equations 32 (2007), 1245-1260.

\bibitem{CS} L. Cafarelli, L. Silvestre, \emph{Regularity theory for fully nonlinear integro-differential equations}, Comm. Pure Appl. Math. 62 (2009), 597-638.

\bibitem{CKS} Z. Chen, P. Kim, R. Song, \emph{Heat kernel estimates for the Dirichlet fractional Laplacian}, J. Eur. Math. Soc. 12 (2010), 1307-1329.

\bibitem{EG} L. C. Evans, R. F. Gariepy, \emph{Measure Theory And Fine Properties Of Functions}, Studies in Advanced Mathematics, CRC Press, 1992.

\bibitem{FW} M.M. Fall, T. Weth, \emph{Nonexistence results for a class of fractional elliptic boundary value problems}, arXiv:1201.4007v1.

\bibitem{G} R. K. Getoor, \emph{First passage times for symmetric stable processes in space}, Trans. Amer. Math. Soc. 101 (1961), 75–90.

\bibitem{GT} D. Gilbarg, N.S. Trudinger, \emph{Elliptic Partial Differential Equations Of Second Order}, Grundlehren der Mathematischen Wissenschaften, Vol. 224. Springer-Verlag, Berlin-New York, 1977.

\bibitem{Kazdan} J. L. Kazdan, \emph{Prescribing The Curvature Of A Riemannian Manifold}, Regional conference series in mathematics 57, American Mathematical Society, 1985.

\bibitem{KL} S. Kim, K. Lee, \emph{Geometric property of the ground state eigenfunction for cauchy process}, arXiv:1105.3283.

\bibitem{Krylov} N. Krylov, \emph{Boundedly inhomogeneous elliptic and parabolic equations in a domain}, Izv. Akad. Nauk SSSR Ser. Mat. 47 (1983), 75–108.

\bibitem{L} N. S. Landkof, \emph{Foundations of Modern Potential Theory}, Springer, New York, 1972.

\bibitem{RS-CRAS} X. Ros-Oton, J. Serra, \emph{Fractional Laplacian: Pohozaev identity and nonexistence results}, C. R. Math. Acad. Sci. Paris 350 (2012), 505-508.

\bibitem{RS} X. Ros-Oton, J. Serra, \emph{The Pohozaev identity for the fractional Laplacian}, preprint arXiv, 2012.

\bibitem{RS-K} X. Ros-Oton, J. Serra, \emph{The boundary Harnack principle for a class of nonlocal equations with bounded measurable coefficients}, in preparation.

\bibitem{SV} O. Savin, E. Valdinoci, \emph{Density estimates for a nonlocal variational model via the Sobolev inequality}, SIAM J. Math. Anal. 43 (2011), 2675–2687.

\bibitem{ServV} R. Servadei, E. Valdinoci, \emph{Mountain pass solutions for non-local elliptic operators}, J. Math. Anal. Appl. 389 (2012), 887-898.

\bibitem{S} L. Silvestre, \emph{Regularity of the obstacle problem for a fractional power of the Laplace operator}, Comm. Pure Appl. Math. 60 (2007), 67-112.

\end{thebibliography}
\end{document}